\documentclass[12pt, reqno]{amsart}
\usepackage{txfonts}      
\usepackage{amssymb}
\usepackage{eucal}
\usepackage{amsmath}
\usepackage{amsthm}
\usepackage{amscd}
\usepackage[dvips]{color}
\usepackage{multicol}
\usepackage[all]{xy}           
\usepackage{graphicx}
\usepackage{color}
\usepackage{colordvi}
\usepackage{xspace}
\usepackage{tikz}
\usepackage{mathrsfs}
\usepackage{ifpdf}
\ifpdf
 \usepackage[colorlinks,final,backref=page,hyperindex]{hyperref}
\else
 \usepackage[colorlinks,final,backref=page,hyperindex,hypertex]{hyperref}
\fi

\newcommand {\emptycomment}[1]{} 

\newtheorem{theorem}{Theorem}[section]
\newtheorem{prop}[theorem]{Proposition}
\newtheorem{lemma}[theorem]{Lemma}
\newtheorem{coro}[theorem]{Corollary}

\theoremstyle{definition}
\newtheorem{defn}[theorem]{Definition}
\newtheorem{remark}[theorem]{Remark}
\newtheorem{exam}[theorem]{Example}

\newtheorem{prop-def}[theorem]{Proposition-Definition}
\newtheorem{coro-def}[theorem]{Corollary-Definition}

\newcommand{\nc}{\newcommand}
\nc{\tred}[1]{\textcolor{red}{#1}}
\nc{\tblue}[1]{\textcolor{blue}{#1}}
\nc{\tgreen}[1]{\textcolor{green}{#1}}
\nc{\tpurple}[1]{\textcolor{purple}{#1}}
\nc{\btred}[1]{\textcolor{red}{\bf #1}}
\nc{\btblue}[1]{\textcolor{blue}{\bf #1}}
\nc{\btgreen}[1]{\textcolor{green}{\bf #1}}
\nc{\btpurple}[1]{\textcolor{purple}{\bf #1}}
\nc{\NN}{{\mathbb N}}

\renewcommand{\frak}{\mathfrak}

\nc{\vsa}{\vspace{-.1cm}} \nc{\vsb}{\vspace{-.2cm}}
\nc{\vsc}{\vspace{-.3cm}} \nc{\vsd}{\vspace{-.4cm}}
\nc{\vse}{\vspace{-.5cm}}

\renewcommand{\textbf}[1]{}

\newcommand{\delete}[1]{}

\nc{\mlabel}[1]{\label{#1}}  
\nc{\mcite}[1]{\cite{#1}}  
\nc{\mref}[1]{\ref{#1}}  
\nc{\meqref}[1]{\eqref{#1}}  
\nc{\mbibitem}[1]{\bibitem{#1}} 

\delete{
\nc{\mlabel}[1]{\label{#1}  
{\hfill \hspace{1cm}{\tt{{\ }\hfill(#1)}}}}
\nc{\mcite}[1]{\cite{#1}{{\tt{{\ }(#1)}}}}  
\nc{\mref}[1]{\ref{#1}{{\tt{{\ }(#1)}}}}  
\nc{\meqref}[1]{\eqref{#1}{{\tt{{\ }(#1)}}}}  
\nc{\mbibitem}[1]{\bibitem[\bf #1]{#1}} 
}

\nc{\opa}{\ast} \nc{\opb}{\odot}  \nc{\pa}{\frakL}
\nc{\arr}{\rightarrow} \nc{\lu}[1]{(#1)} \nc{\mult}{\mrm{mult}}
\nc{\diff}{\mathfrak{Diff}}
\nc{\opc}{\sharp}\nc{\opd}{\natural}
\nc{\dpt}{\mathrm{d}}
\nc{\tforall}{\text{ for all }}
\nc{\diam}{alternating\xspace}
\nc{\Diam}{Alternating\xspace}
\nc{\cdiam}{alternating\xspace}
\nc{\Cdiam}{Alternating\xspace}
\nc{\AW}{\mathcal{A}}
\nc{\rba}{Rota-Baxter algebra\xspace}

\nc{\ari}{\mathrm{ar}}

\nc{\lef}{\mathrm{lef}}

\nc{\Sh}{\mathrm{ST}}

\nc{\Cr}{\mathrm{Cr}}

\nc{\st}{{Schr\"oder tree}\xspace}
\nc{\sts}{{Schr\"oder trees}\xspace}

\nc{\vertset}{\Omega} 

\nc{\pb}{{\mathrm{pb}}}
\nc{\Lf}{{\mathrm{Lf}}}

\nc{\lft}{{left tree}\xspace}
\nc{\lfts}{{left trees}\xspace}

\nc{\fat}{{fundamental averaging tree}\xspace}

\nc{\fats}{{fundamental averaging trees}\xspace}
\nc{\avt}{\mathrm{Avt}}

\nc{\rass}{{\mathit{RAss}}}

\nc{\aass}{{\mathit{AAss}}}

\nc{\vin}{{\mathrm Vin}}    
\nc{\lin}{{\mathrm Lin}}    
\nc{\inv}{\mathrm{I}n}
\nc{\gensp}{V} 
\nc{\genbas}{\mathcal{V}} 
\nc{\bvp}{V_P}     
\nc{\gop}{{\,\omega\,}}     

\nc{\bin}[2]{ (_{\stackrel{\scs{#1}}{\scs{#2}}})}  
\nc{\binc}[2]{ \left (\!\! \begin{array}{c} \scs{#1}\\
    \scs{#2} \end{array}\!\! \right )}  
\nc{\bincc}[2]{  \left ( {\scs{#1} \atop
    \vspace{-1cm}\scs{#2}} \right )}  
\nc{\bs}{\bar{S}} \nc{\cosum}{\sqsubset} \nc{\la}{\longrightarrow}
\nc{\rar}{\rightarrow} \nc{\dar}{\downarrow} \nc{\dprod}{**}
\nc{\dap}[1]{\downarrow \rlap{$\scriptstyle{#1}$}}
\nc{\md}{\mathrm{dth}} \nc{\uap}[1]{\uparrow
\rlap{$\scriptstyle{#1}$}} \nc{\defeq}{\stackrel{\rm def}{=}}
\nc{\disp}[1]{\displaystyle{#1}} \nc{\dotcup}{\
\displaystyle{\bigcup^\bullet}\ } \nc{\gzeta}{\bar{\zeta}}
\nc{\hcm}{\ \hat{,}\ } \nc{\hts}{\hat{\otimes}}
\nc{\barot}{{\otimes}} \nc{\free}[1]{\bar{#1}}
\nc{\uni}[1]{\tilde{#1}} \nc{\hcirc}{\hat{\circ}} \nc{\lleft}{[}
\nc{\lright}{]} \nc{\lc}{\lfloor} \nc{\rc}{\rfloor}
\nc{\curlyl}{\left \{ \begin{array}{c} {} \\ {} \end{array}
    \right .  \!\!\!\!\!\!\!}
\nc{\curlyr}{ \!\!\!\!\!\!\!
    \left . \begin{array}{c} {} \\ {} \end{array}
    \right \} }
\nc{\longmid}{\left | \begin{array}{c} {} \\ {} \end{array}
    \right . \!\!\!\!\!\!\!}
\nc{\onetree}{\bullet} \nc{\ora}[1]{\stackrel{#1}{\rar}}
\nc{\ola}[1]{\stackrel{#1}{\la}}
\nc{\ot}{\otimes} \nc{\mot}{{{\boxtimes\,}}}
\nc{\otm}{\overline{\boxtimes}} \nc{\sprod}{\bullet}
\nc{\scs}[1]{\scriptstyle{#1}} \nc{\mrm}[1]{{\rm #1}}
\nc{\margin}[1]{\marginpar{\rm #1}}   
\nc{\dirlim}{\displaystyle{\lim_{\longrightarrow}}\,}
\nc{\invlim}{\displaystyle{\lim_{\longleftarrow}}\,}
\nc{\mvp}{\vspace{0.3cm}} \nc{\tk}{^{(k)}} \nc{\tp}{^\prime}
\nc{\ttp}{^{\prime\prime}} \nc{\svp}{\vspace{2cm}}
\nc{\vp}{\vspace{8cm}} \nc{\proofbegin}{\noindent{\bf Proof: }}
\nc{\proofend}{$\blacksquare$ \vspace{0.3cm}}
\nc{\modg}[1]{\!<\!\!{#1}\!\!>}
\nc{\intg}[1]{F_C(#1)} \nc{\lmodg}{\!
<\!\!} \nc{\rmodg}{\!\!>\!}
\nc{\cpi}{\widehat{\Pi}}
\nc{\sha}{{\mbox{\cyr X}}}  
\nc{\shap}{{\mbox{\cyrs X}}} 
\nc{\shan}{{\overrightarrow \sha}}
\nc{\shpr}{\diamond}    
\nc{\shp}{\ast} \nc{\shplus}{\shpr^+}
\nc{\shprc}{\shpr_c}    
\nc{\msh}{\ast} \nc{\zprod}{m_0} \nc{\oprod}{m_1}
\nc{\vep}{\varepsilon} \nc{\labs}{\mid\!} \nc{\rabs}{\!\mid}
\nc{\sqmon}[1]{\langle #1\rangle}

\nc{\mmbox}[1]{\mbox{\ #1\ }} \nc{\fp}{\mrm{FP}}
\nc{\rchar}{\mrm{char}} \nc{\End}{\mrm{End}} \nc{\Fil}{\mrm{Fil}}
\nc{\Mor}{Mor\xspace} \nc{\gmzvs}{gMZV\xspace}
\nc{\gmzv}{gMZV\xspace} \nc{\mzv}{MZV\xspace}
\nc{\mzvs}{MZVs\xspace} \nc{\Hom}{\mrm{Hom}} \nc{\id}{\mrm{id}}
\nc{\im}{\mrm{im}} \nc{\incl}{\mrm{incl}} \nc{\map}{\mrm{Map}}
\nc{\mchar}{\rm char} \nc{\nz}{\rm NZ} \nc{\supp}{\mathrm Supp}

\nc{\Alg}{\mathbf{Alg}} \nc{\Bax}{\mathbf{Bax}} \nc{\bff}{\mathbf f}
\nc{\bfk}{{\bf k}} \nc{\bfone}{{\bf 1}} \nc{\bfx}{\mathbf x}
\nc{\bfy}{\mathbf y}
\nc{\base}[1]{\bfone^{\otimes ({#1}+1)}} 
\nc{\Cat}{\mathbf{Cat}}

\nc{\detail}{\marginpar{\bf More detail}
    \noindent{\bf Need more detail!}
    \svp}
\nc{\Int}{\mathbf{Int}} \nc{\Mon}{\mathbf{Mon}}
\nc{\rbtm}{{shuffle }} \nc{\rbto}{{Rota-Baxter }}
\nc{\remarks}{\noindent{\bf Remarks: }} \nc{\Rings}{\mathbf{Rings}}
\nc{\Sets}{\mathbf{Sets}} \nc{\wtot}{\widetilde{\odot}}
\nc{\wast}{\widetilde{\ast}} \nc{\bodot}{\bar{\odot}}
\nc{\bast}{\bar{\ast}} \nc{\hodot}[1]{\odot^{#1}}
\nc{\hast}[1]{\ast^{#1}} \nc{\mal}{\mathcal{O}}
\nc{\tet}{\tilde{\ast}} \nc{\teot}{\tilde{\odot}}
\nc{\oex}{\overline{x}} \nc{\oey}{\overline{y}}
\nc{\oez}{\overline{z}} \nc{\oef}{\overline{f}}
\nc{\oea}{\overline{a}} \nc{\oeb}{\overline{b}}
\nc{\weast}[1]{\widetilde{\ast}^{#1}}
\nc{\weodot}[1]{\widetilde{\odot}^{#1}} \nc{\hstar}[1]{\star^{#1}}
\nc{\lae}{\langle} \nc{\rae}{\rangle}
\nc{\lf}{\lfloor}
\nc{\rf}{\rfloor}

\nc{\QQ}{{\mathbb Q}}
\nc{\RR}{{\mathbb R}} \nc{\ZZ}{{\mathbb Z}}
\nc{\CC}{{\mathbb C}}

\nc{\cala}{{\mathcal A}} \nc{\calb}{{\mathcal B}}
\nc{\calc}{{\mathcal C}}
\nc{\cald}{{\mathcal D}} \nc{\cale}{{\mathcal E}}
\nc{\calf}{{\mathcal F}} \nc{\calg}{{\mathcal G}}
\nc{\calh}{{\mathcal H}} \nc{\cali}{{\mathcal I}}
\nc{\call}{{\mathcal L}} \nc{\calm}{{\mathcal M}}
\nc{\caln}{{\mathcal N}}\nc{\calo}{{\mathcal O}}
\nc{\calp}{{\mathcal P}} \nc{\calq}{\mathcal{Q}} \nc{\calr}{{\mathcal R}}
\nc{\cals}{{\mathcal S}} \nc{\calt}{{\mathcal T}}
\nc{\calu}{{\mathcal U}} \nc{\calw}{{\mathcal W}} \nc{\calk}{{\mathcal K}}
\nc{\calx}{{\mathcal X}} \nc{\CA}{\mathcal{A}}

\nc{\fraka}{{\mathfrak a}} \nc{\frakA}{{\mathfrak A}}
\nc{\frakb}{{\mathfrak b}} \nc{\frakB}{{\mathfrak B}}
\nc{\frakD}{{\mathfrak D}} \nc{\frakF}{\mathfrak{F}}
\nc{\frakf}{{\mathfrak f}} \nc{\frakg}{{\mathfrak g}}
\nc{\frakH}{{\mathfrak H}} \nc{\frakL}{{\mathfrak L}}
\nc{\frakM}{{\mathfrak M}} \nc{\bfrakM}{\overline{\frakM}}
\nc{\frakm}{{\mathfrak m}} \nc{\frakP}{{\mathfrak P}}
\nc{\frakN}{{\mathfrak N}} \nc{\frakp}{{\mathfrak p}}
\nc{\frakS}{{\mathfrak S}} \nc{\frakT}{\mathfrak{T}}
\nc{\frakX}{{\mathfrak X}} \nc{\frakx}{\mathfrak{x}}
\nc{\frakc}{{\mathfrak c}}
\nc{\frakd}{{\mathfrak d}}
\nc{\frake}{{\mathfrak e}}
\nc{\BS}{\mathbb{S}}

\font\cyr=wncyr10 \font\cyrs=wncyr7

\nc{\mapmonoid}{\frakM}
\nc{\ncsha}{{\mbox{\cyr X}^{\mathrm NC}}} 
\nc{\ncshao}{{\mbox{\cyr X}^{\mathrm NC}}}
\nc{\ce}{{\mbox{\cyr X}_\frake^{\mathrm NC}}(A)}
\nc{\ced}{{\mbox{\cyr X}_\frake^{\mathrm NC}}(T^+(D))}
\nc{\cet}{{\mbox{\cyr X}_\frake^{\mathrm NC}}(T)}
\nc{\dfgen}{V} \nc{\dfrel}{\Lambda}
\nc{\dfgenb}{\vec{v}} \nc{\dfrelb}{\vec{r}}
\nc{\dfgene}{v} \nc{\dfrele}{r}
\nc{\dfop}{\odot}
\nc{\dfoa}{\dfop^{(1)}} \nc{\dfob}{\dfop^{(2)}}
\nc{\dfoc}{\dfop^{(3)}} \nc{\dfod}{\dfop^{(4)}}
\nc{\erba}{\mathbf{ERBA}}
\nc{\ed}{\mathbf{ED}}
\nc{\etd}{\mathbf{ETD}}
\nc{\NS}{\mathbf{NS}}
\nc{\FN}{F_{\mathrm N}}
\nc{\ob}{\ \begin{picture}(-1,1)(-1,-3)\circle*{3}\end{picture}\ \,}
\nc{\oc}{\circ}
\nc{\obp}{{\ \begin{picture}(-1,1)(-1,-3)\circle*{3}\end{picture}\  }_P}
\nc{\ocp}{\circ_P}
\nc{\pr}{\prec}
\nc{\su}{\succ}
\nc{\prp}{\prec_P}
\nc{\scp}{\succ_P}
\nc{\UN}{U_{N}}
\nc{\denshpr}{\den{\shpr}}
\nc{\den}[1]{\check{#1}}
\nc{\freea}[1]{\tilde{#1}}
\nc{\freev}[1]{\hat{#1}}
\nc{\speciall}{\mathrm{sl(2,\mathbb{C})}}

\nc{\vd}{\vdash}
\nc{\dv}{\dashv}
\nc{\dybe}{DYBE\xspace}
\nc{\da}{diassociative algebra\xspace}
\nc{\dca}{diassociative coalgebra\xspace}
\nc{\das}{diassociative algebras\xspace}
\nc{\dabi}{diassociative bialgebra\xspace}
\nc{\dabis}{diassociative bialgebras\xspace}
\nc{\pdas}{pre-diassociative algebras\xspace}
\nc{\pda}{pre-diassociative algebra\xspace}

\nc{\vda}{\vd_A}
\nc{\vds}{\vd_{A^*}}
\nc{\vdbs}{\vd_{B^*}}
\nc{\vdb}{\vd_B}
\nc{\dva}{\dashv_A}
\nc{\dvs}{\dv_{A^*}}
\nc{\dvbs}{\dv_{B^*}}
\nc{\dvb}{\dv_B}
\nc{\oba}{\ob_A}
\nc{\obas}{\ob_{A^*}}

\definecolor{darkred}{rgb}{0.7,0,0} 

\definecolor{darkgreen}{RGB}{0,180,0}

\begin{document}

\title[A bialgebra theory for diassociative algebras]{Bialgebra theory, the Yang-Baxter equation and relative Rota-Baxter operators for diassociative algebras}
%

\author{Hui Hu}
\address{School of Mathematics and Statistics, Jiangxi Normal University, Nanchang, Jiangxi 330022, China}
\email{huhui@jxnu.edu.cn}

\author{Yizhen Li}
\address{School of Mathematics and Statistics, Jiangxi Normal University, Nanchang, Jiangxi 330022, China}
\email{liyz@jxnu.edu.cn}

\author{Guilai Liu}
\address{Chern Institute of Mathematics \& LPMC, Nankai University, Tianjin 300071, China
	\newline
	Institut de Recherche Mathématique Avancée, UMR7501, Université de Strasbourg, Strasbourg Cedex 67084, France}	 \email{liugl@mail.nankai.edu.cn}

\author{Shanghua Zheng}
\address{School of Mathematics and Statistics, Jiangxi Normal University, Nanchang, Jiangxi 330022, China}
\email{zhengsh@jxnu.edu.cn}

\date{\today}
\begin{abstract}
	In this paper, we develop a bialgebra theory for diassociative algebras.
	Inspired by the notion of a quadratic diassociative algebra, we introduce the concept of a Manin triple of diassociative algebras. We then define a diassociative bialgebra, which is shown to be equivalent to a Manin triple of diassociative algebras through a specific matched pair of diassociative algebras. 
	We further formulate the diassociative Yang-Baxter equation (DYBE) in a diassociative algebra, and prove that symmetric solutions of the DYBE give rise to diassociative bialgebras. 
To construct such solutions, we also introduce relative Rota-Baxter operators and pre-diassociative algebras.
 As a key application, we lift the known relationships between diassociative algebras and other algebraic structures to the bialgebra level. In particular, we show that every diassociative bialgebra naturally induces a Leibniz bialgebra, thereby extending Loday’s classical result that a diassociative algebra gives rise to a Leibniz algebra. Moreover, we provide explicit constructions of Lie bialgebras via tensor products of diassociative bialgebras and  quadratic dendriform  algebras.
 
\end{abstract}
\makeatletter
\@namedef{subjclassname@2020}{\textup{2020} Mathematics Subject Classification}
\makeatother
\subjclass[2020]{
	17B38, 
	17A30, 
	16W99, 
	16T10, 
17B62 
}

\keywords{Diassociative algebra, bialgebra, Manin triple, Yang-Baxter equation, relative Rota-Baxter operator}

\maketitle

\vspace{-1cm}

\tableofcontents

\vspace{-1cm}

\allowdisplaybreaks

\section{Introduction}
The purpose of this paper is threefold: to establish a bialgebra theory for diassociative algebras via the Manin triple approach, to construct diassociative bialgebras using the Yang-Baxter equation and relative Rota-Baxter operators, and to investigate the relationships between diassociative bialgebras and other types of bialgebras, such as Leibniz and Lie bialgebras.

\subsection{A brief review of diassociative algebras}
Diassociative  algebras were first introduced by Loday in his study of the periodicity phenomena in algebraic $K$-theory~\cite{Lo1}.
\begin{defn} A {\bf diassociative algebra} (also called an {\bf associative dialgebra}) is a triple $(A, \dva, \vda)$ consisting of a vector space $A$ and  binary operations  $\dva,\vda: A\ot A\to A$ such that for all $x,y,z\in A$, the following equations hold:
	\begin{align}
		(x \dva y) \dva z&=x \dva (y \dva z)=  x \dva (y \vda z) ,\mlabel{eq:dia1}\\
		(x \vda y) \dva z&=   x \vda (y \dva z),\mlabel{eq:dia2}\\
		(x \vda y) \vda z &=(x \dva y) \vda z=  x \vda (y \vda z).\mlabel{eq:dia3}
	\end{align}
\end{defn}

For a diassociative algebra $(A, \dva)$, if equation $x \dva y = y \vda x$ holds for all $x,y\in A$, then $(A, \dva)$ becomes a perm algebra~\cite{Cha01}. Thus diassociative algebras serve as the noncommutative analogues of perm algebras.
On the other hand,
 a diassociative algebra “doubles” the associative algebra in the sense that it has two associative operations with certain compatible conditions. From the viewpoint of operad theory, the operad of diassociative algebras is the duplicator of the operad of associative algebras \cite{PBGN}, which leads to the classical construction of diassociative algebras  from averaging associative algebras. To be more precise, 
an {\bf averaging associative algebra} is a triple $(A,\cdot,P)$ consisting of an associative algebra $(A,\cdot)$ and a linear operator $P$ such that
	\begin{equation*}
		P(x)\cdot P(y)=P(x\cdot P(y))=P(P(x)\cdot y),\quad x,y\in A.
	\end{equation*}
	Such an operator $P$ is called an {\bf averaging operator} on $(A,\cdot)$.
 Define binary operations on $A$ by
\begin{equation*}
x\dva y:=x\cdot P(y),\quad x \vda y:=P(x)\cdot y,\quad x,y\in A.
\end{equation*}
	Then by \cite[Proposition~6.2]{Ag1}, $(A,\dva,\vda)$ is a diassociative algebra.
See also ~\cite[P16, Examples (b)]{Lo1} for this construction from a differential algebra $(A,\cdot,d)$ with $d^2=0$, since a differential operator $d$ with $d^2=0$ is an averaging operator~\cite{Ag1}.

Diassociative algebras have garnered considerable attention. For example, motivated by the connection between diassociative algebras and conformal algebras, Kolesnikov~\cite{Kol08} provided a general scheme for identifying the diassociative algebra variety corresponding to a given variety of ordinary algebras, such as associative, Lie, or alternative algebras. 
In subsequent work, Saha~\cite{Saha19} developed an equivariant cohomology theory for diassociative algebras and applied it to study their formal deformations. More recently, Das and Sen~\cite{DS23} introduced diassociative family algebras and showed that averaging family operators induce such structures. The latest contribution by Restrepo-S\'anchez  et al.~\cite{RRS25} equips \( F[x] \otimes F[x] \) with a diassociative algebra structure and offers a complete classification of its derivations and diderivations.  For additional work on classifications and derivations of diassociative algebras, see \cite{Gon13,RRB14}, and a comprehensive survey of diassociative algebras can be found in \cite{Bre12}.

Diassociative algebras are closely related to several other types of algebras, including Leibniz algebras and dendriform algebras.  
The concept of a Leibniz algebra was first introduced by Blokh, who referred to it as a \(D\)-algebra~\cite{Blo65}. As a noncommutative analogue of Lie algebras, the Leibniz algebra was later rediscovered and popularized (see \cite{Lo93,LP93}) by Loday, who coined the term (right) Leibniz algebra.  

Recall from \cite{Lo1} that a {\bf (left) Leibniz algebra} is a pair \((A, \circ_A)\), where \(A\) is a vector space and \(\circ_A: A \otimes A \to A\) is a bilinear operation satisfying the {\bf (left) Leibniz identity}
$$
x \circ_A (y \circ_A z) = (x \circ_A y) \circ_A z + y \circ_A (x \circ_A z), \quad x, y, z \in A.
$$

Every diassociative algebra gives rise to a Leibniz algebra in the following way, generalizing the canonical correspondence between associative algebras and Lie algebras.

\begin{exam} \cite{Lo1,RRS25}\mlabel{exam:dia-Lei}
	Let $(A, \vda, \dva)$ be a diassociative algebra. Then  the vector space $A$ admits a natural Leibniz algebra structure $(A, \circ_{A})$ with the multiplication $\circ_{A}$ given by
		\begin{equation}\mlabel{eq:dil}
			x\circ_{A} y:=x\vda y-y\dva x,\quad x,y\in A.
		\end{equation}
\end{exam}

Recall from \cite{Ag1,BGLZ251} that an {\bf averaging Lie algebra} is a triple $(\mathfrak{g}, [-,-]_\mathfrak{g}, P)$, where $(\mathfrak{g},[-,-]_\mathfrak{g})$ is a Lie algebra and $P \colon \mathfrak{g} \to \mathfrak{g}$ is a linear map satisfying the identity
$$
[P(x),P(y)]_\mathfrak{g} = P\bigl([P(x),y]_\mathfrak{g}\bigr), \quad x,y \in \mathfrak{g}.
$$
Paralleling the construction of diassociative algebras via averaging operators on associative algebras,
an averaging Lie algebra $(\mathfrak{g}, [-,-]_\frakg, P)$  induces a Leibniz algebra structure $(\mathfrak{g},\circ_{\frak g})$, with the binary operation $\circ_\mathfrak{g}$ defined by
$$
x \circ_{\frak g} y \;=\; [P(x),y]_\mathfrak{g}, \quad  x,y \in \mathfrak{g}.
$$

To encapsulate the structural correspondences established thus far, we present the following commutative diagram:
$$\small{
	\xymatrix{
		\text{Associative algebras}
		\ar@{->}_{\text{averaging operator}}[d]\ar@{->}^{\quad[x,y]_{A}:=x\cdot_{A} y-y\cdot_{A} x}[r] 
		&\text{Lie algebras}\ar@{->}^{\text{averaging operator}}[d]\\
		\text{Diassociative algebras\quad}\ar@{->}^{x\circ_{A} y:=x\vda y- y\dva x}[r]&\text{\quad Leibniz algebras}
}}
$$

	On the other hand, diassociative algebras are also closely related to dendriform algebras.	
The concept of dendriform algebras was introduced by Loday in the context of algebraic $K$-theory. They arise in numerous areas of mathematics and physics, including combinatorics~\cite{LR2}, operads~\cite{Lo1}, homology theory~\cite{Fra97,Fra98},  pre-Lie algebras~\cite{Bai21,Cha01}, Hopf algebras~\cite{Cha02,Ro1}, Rota-Baxter algebras~\cite{EG1},   Lie and Leibniz algebras~\cite{Fra98}, and quantum field theory~\cite{Foi02}.  A key feature of dendriform algebras is that the sum of their two binary operations for a dendriform algebra $(A, \succ_{A}, \prec_{A})$ yields an associative algebra $(A, \cdot_{A})$. This phenomenon can be understood as a form of ``splitting associativity".
The operad of dendriform algebras is Koszul dual to the operad of diassociative algebras. Consequently, one obtains a natural construction of Lie algebras from tensor products of dendriform and diassociative algebras.

\subsection{Diassociative bialgebras, the diassociative Yang-Baxter equation and relative Rota-Baxter operators}

A bialgebra structure consists of a vector space equipped with both an algebra and a coalgebra structure, which satisfy a set of compatibility conditions. 
Some known examples of such structures
include Lie bialgebras (\cite{CP94,Dri})  and antisymmetric
infinitesimal bialgebras (\cite{Agu2000, Agu2001,
	Agu2004,Bai}).
These bialgebras have a common
property that they have equivalent characterizations in terms of the Manin triples which correspond to the nondegenerate (symmetric or antisymmetric) invariant bilinear forms.

In \cite{Cha05}, Chapoton proved that the operads of pre-Lie algebras, Zinbiel algebras, dendriform algebras, Leibniz algebras, perm algebras and diassociative algebras are anticyclic. He also applied operad theory to properly determine the nondegenerate antisymmetric invariant bilinear forms on these algebras. 
To date, bialgebra theories have been established  for 
pre-Lie algebras (see \cite{Bai1,Wang2024}), Zinbiel algebras (see \cite{Wang2026}), dendriform algebras (see \cite{Bai,Wang25}), Leibniz algebras (see \cite{TS22,BLST25}) and perm algebras (see \cite{Hou24, Lin}).
However, a systematic bialgebra theory for diassociative algebras has not yet been developed.

In this paper, we introduce the notion of a diassociative bialgebra. We show that a diassociative bialgebra is equivalent to a Manin triple of diassociative algebras through specific matched pairs of diassociative algebras.
Then we introduce the diassociative Yang-Baxter equation (DYBE) in a diassociative algebra, and show that symmetric solutions of the DYBE give rise to diassociative bialgebras. 
We also introduce the notions of relative Rota-Baxter operators and pre-diassociative algebras to construct solutions of the DYBE.

It is natural to lift the aforementioned algebraic correspondences between diassociative algebras and other varieties to the bialgebra level.
We show that a diassociative bialgebra naturally gives rise to a Leibniz bialgebra.
We also show that the tensor product of a diassociative bialgebra and a quadratic dendriform algebra yields a Lie bialgebra. 
There is also a construction of  diassociative bialgebras from tensor products of quadratic perm algebras and ASI bialgebras.

\subsection{Outline of the paper} The organization of this paper is as follows.
	
In Section~\mref{sec:2}, we  recall some examples of diassociative algebras, and give the notion of a diassociative coalgebra. We then develop the representation theory of a \da via its semi-direct product structure. We prove that every representation of a \da has a corresponding dual representation (Proposition~\mref{prop:dual-rep}). Consequently, the adjoint representation of a \da also admits a dual representation.

In Section~\mref{sec:3}, we first define matched pairs and Manin triples of \das. We then establish that a distinguished class of matched pairs is equivalent to a Manin triple of \das (Theorem~\mref{thm:Manin}). Building on this equivalence, we further introduce the notion of a \dabi, and subsequently demonstrate the mutual equivalence among matched pairs, Manin triples of \das and \dabis (Corollary~\mref{coro:equiv}).

In Section~\mref{sec:4}, we give the notion of the diassociative Yang-Baxter Equation (DYBE) in a \da. We first demonstrate that a symmetric solution of the DYBE yields  a diassociative bialgebra (Proposition~\mref{prop:triangular}). 
 Next, we propose an $\mathcal{O}$-operator for a \da as an operator form of the DYBE, and prove that each such operator induces a symmetric solution of the DYBE in a suitable semi-direct product diassociative algebra (Theorem~\mref{thm:semi-dybe}).
 Finally, we introduce pre-diassociative algebras to systematically construct the DYBE solutions (Theorem~\mref{thm:pre-dia}).

In Section~\mref{sec:5}, we first show that diassociative bialgebras induce Leibniz bialgebras through the Manin triple approach (Theorem~\mref{thm:dabi-leibi}). This result lifts Loday's celebrated correspondence from diassociative algebras to Leibniz algebras to the bialgebra level. We then prove that a Lie bialgebra can be derived from the tensor product of a quadratic dendriform algebra and a diassociative bialgebra (Proposition~\mref{prop:dend-dia-Lie}). Finally, we construct diassociative bialgebras from tensor products of quadratic perm algebras and ASI bialgebras  (Proposition~\mref{prop:ass-perm-dia}). These results highlight the pivotal role of diassociative bialgebras in connecting various bialgebraic structures.	
	
{\bf Notations: }
\begin{enumerate}
\item
In this paper, the ground field is denoted by $\bfk$, which is assumed to be of characteristic 0. All vector spaces, tensor products, and linear homomorphisms are defined over $\bfk$. Unless specified otherwise, all vector spaces and algebras considered are finite-dimensional. By an algebra, we mean an associative algebra, not necessarily unital. For a vector space $V$, we denote its dual space by $V^*$ .
\item
Let $V$ be a vector space equipped with a binary operation $\diamond_{V}$. For all $u, v \in V$, we define the {\bf left multiplication operator} $L_{\diamond_{V}}(u)$ (or simply $L(u)$) by  
$$
L_{\diamond_{V}}(u)(v) = u \diamond_{V} v,
$$
and the {\bf right multiplication operator} $R_{\diamond_{V}}(v)$ (or simply $R(v)$) by  
$$
R_{\diamond_{V}}(v)(u) = u \diamond_{V} v.
$$
The {\bf flip map} $\sigma:V\ot V\to V\ot V$ is the linear map given by $ \sigma(u\ot v)=v\ot u$.
\item For any vector space $V$ and its dual space $V^*$, the natural pairing is denoted by  
 $$
   \langle \cdot, \cdot \rangle : V^* \times V \to \bfk, \quad \langle f, v \rangle = f(v), \quad f \in V^*, v \in V.
$$
   Given a linear map \(\varphi: V \to W\), its transpose \(\varphi^*: W^* \to V^*\) is defined by  
$$
   \langle \varphi^*(w^*), v \rangle = \langle w^*, \varphi(v) \rangle, \quad  w^* \in W^*, v \in V.
$$

\item Let $A$ and $V$ be vector spaces, and let $\rho: A \to \mathrm{End}(V)$ be a linear map. The dual representation $\rho^*: A \to \mathrm{End}(V^*)$ is defined by  
\begin{equation}\mlabel{eq:rho-star}
   \langle \rho^*(x)(u^*), v \rangle  
   =\langle \big( \rho(x) \big)^{*}u^{*},v\rangle=
    \langle u^*, \rho(x)(v) \rangle, \quad  x \in A, u^* \in V^*, v \in V.
\end{equation}
\end{enumerate}
\vspace{-.4cm}

\section{Diassociative algebras and their representations}
\mlabel{sec:2}
In this section, we first recall several examples of diassociative algebras  together with the definition of a quadratic \da. We then present their representation theory, and show that the dual space of a representation of a \da also carries a natural representation structure.

\subsection{Diassociative algebras}

\begin{exam}\cite{Ongay07}
Let $V$ be a vector space and let $\phi\in V^*$  be fixed. Define
 $$u\vd v:=\phi(u)v,\quad u\dv v:=u\phi(v),\quad u,v\in V.$$
Since   $\phi(\phi(u)v)=\phi(u)\phi(v)=\phi(u\phi(v))$, $(V,\vd,\dv)$ is a diassociative algebra.
\end{exam}
\begin{exam}\cite{LZ10} Let $\bfk[x,y]$ be the polynomial algebra in two variables $x$ and $y$ over  $\bfk$. 
Define 
$$f(x,y)\dv g(x,y):=f(x,y)g(y,y),\; f(x,y)\vd g(x,y):=f(x,x)g(x,y),\; f(x,y), g(x,y)\in\bfk[x,y].$$
Then $(\bfk[x,y],\dv,\vd)$ forms a diassociative algebra.
\end{exam}
\begin{exam} 
Let $V$ be a two-dimensional vector space with basis $\{e_1, e_2\}$.
Define two multiplications $\dv$ and $\vd$ on $V$ by the following Cayley tables.
$$
		\begin{minipage}[t]{0.3\textwidth}
			\centering
			\begin{tabular}{c|cc}
				$\dv$	& $e_1$ &$ e_2$\\
				\hline 	
				$e_1$ & $e_1$ & $e_1$   \\
				$e_2$ & $e_2$ & $e_2$  \\
			\end{tabular}
		\end{minipage}
		\hspace{0.02cm}
		\begin{minipage}[t]{0.3\textwidth}
			\centering
				\begin{tabular}{c|cc}
					$\vd$	& $e_1$ & $e_2$  \\
					\hline 	
					$e_1$ & $e_1$ & $ e_2$ \\
					$e_2$ & $e_1$ & $e_2$ \\
			\end{tabular}
		\end{minipage}
$$
A straightforward verification shows that $(V,\dv,\vd)$ is a \da.
\end{exam}

\begin{defn}\cite{Cha05} \mlabel{defn:qua-dia}
	A {\bf quadratic  diassociative algebra}  is a diassociative algebra $(A,\dva,\vda)$  equipped with a
	nondegenerate antisymmetric bilinear form $\omega$  that is {\bf invariant} in the following sense:
	\begin{equation}\mlabel{eq:omega}
		\omega(x,y\dva z)=\omega(z, x\oba y),\;
		\;\omega(x\vda y,z)=\omega(x,y\oba z),\quad x,y,z\in A,
	\end{equation}
	where $x\oba y:=x\vd_A y-x\dv_A y$.
\end{defn}

\begin{remark}
Note that ~\meqref{eq:omega} is equivalent to the following two equations:
	\begin{equation}\mlabel{eq:omega1}
		\omega(x\dva y, z)=-\omega(y\vda z,x),	
\end{equation}
\begin{equation}\mlabel{eq:omega2}
	\omega(x\vda y,z)=\omega(y\dva z,x)-\omega(y\vda z,x),\quad x,y,z\in A.
	\end{equation}
Moreover, under ~\meqref{eq:omega1}, ~\meqref{eq:omega2} is equivalent to
\begin{equation*}
	\omega(x\vda y,z)=-\omega(z\vda x,y)-\omega(y\vda z,x),\quad x,y,z\in A.
\end{equation*}
Thus $\omega$ is a Connes cocycle \cite{Bai} on the associative algebra $(A,\vda)$, and likewise on the associative algebra $(A, \dva)$. Therefore, any quadratic diassociative algebra induces a Connes cocycle.
\end{remark}
We now present the dual notion of diassociative algebras.
\begin{defn}\mlabel{defn:dia-coalg}
	A {\bf diassociative  coalgebra} is a triple $(A,\Delta_{\dva},\Delta_{\vda})$ consisting of a vector space $A$ and co-multiplications $\Delta_{\dva},\Delta_{\vda}:A\rightarrow A\otimes A$ such that for all $x\in A$,
	\begin{align}
		&(\id\otimes\Delta_{\dva})\Delta_{\dva}(x)=(\id\otimes\Delta_{\vda})\Delta_{\dva}(x)=(\Delta_{\dva}\otimes\id )\Delta_{\dva}(x),\mlabel{eq:co1}\\
		&(\Delta_{\vda}\otimes\id )\Delta_{\dva}(x)=(\id\otimes\Delta_{\dva})\Delta_{\vda}(x),\mlabel{eq:co2}\\
		&(\Delta_{\dva}\otimes\id )\Delta_{\vda}(x)=(\Delta_{\vda}\otimes\id )\Delta_{\vda}(x)=(\id\otimes\Delta_{\vda})\Delta_{\vda}(x).\mlabel{eq:co3}
	\end{align}
\end{defn}
\begin{exam}
 Let $A$ be a two-dimensional vector space with basis $\{e_1, e_2\}$. Define two coproducts $\Delta_{\dva}$ and $\Delta_{\vda}$ by
\begin{align*}\Delta_{\dva}(e_1)=e_1\ot e_1,\qquad &\Delta_{\dva}(e_2)=e_2\ot e_1.\\
\Delta_{\vda}(e_1)=e_1\ot e_1, \qquad&\Delta_{\vda}(e_2)=0.
\end{align*}
Thus, a straightforward calculation shows that $(A, \Delta_{\dva}, \Delta_{\vda})$ is a \dca.
\end{exam}

The following result establishes the precise duality between diassociative algebras and coalgebras.
\begin{prop}\mlabel{prop:dia-coalg}
Let $A$ be a vector space and let $\Delta_{\dva}$ and $\Delta_{\vda}$ be co-multiplications on $A$. Let $\dashv_{A^*}$ and $\vds$ be the linear duals of $\Delta_{\dva}$ and $\Delta_{\vda}$, respectively. Then
$(A, \Delta_{\dva}, \Delta_{\vda})$ is a  diassociative  coalgebra if and only if $(A^*,\dashv_{A^*}, \vds)$ is a diassociative algebra.  
\end{prop}

\begin{exam}\mlabel{defn:aver-coalg}
An {\bf averaging coalgebra} is a triple $(C,\Delta, Q)$ consisting of a coassociative coalgebra $(C,\Delta)$ and a linear operator $Q$ such that
	\begin{equation*}
	(Q\ot Q)\Delta(x)=(\id\ot Q)\Delta( Q(x))= (Q\ot \id)\Delta (Q(x)) ,\quad x\in C.
	\end{equation*}
An  averaging coalgebra  $(C,\Delta, Q)$ induces a diassociative coalgebra $(C, \Delta_{\dv_C},\Delta_{\vd_C})$, where $\Delta_{\dv_C}$ and $\Delta_{\vd_C}$ are given by
\begin{equation*}
\Delta_{\dv_C}(x):=(\id\ot Q)\Delta(x),\quad \Delta_{\vd_C}(x):=(Q\ot\id)\Delta(x),\quad x\in C.
\end{equation*}
\end{exam}
\subsection{Representations of diassociative algebras}
\begin{defn}\mlabel{def:semi}
	Let $(A,\dva,\vda)$ be a \da and let $V$ be a vector space. Let $\ell_{\dva},r_{\dva},\ell_{\vda},r_{\vda}: A\to \End(V)$ be linear maps.
	Define two multiplications $\dv',\vd'$ on $A\oplus V$ by
	\begin{align}
		&(a+u)\dv' (b+v):=a\dva b+(\ell_{\dva}(a)v+r_{\dva}(b)u ),\\
		&(a+u)\vd' (b+v):=a\vda b+(\ell_{\vda}(a)v+r_{\vda}(b)u ),\quad  a, b\in A, u, v \in V.
	\end{align}
	If ($A\oplus V$,$\dashv'$,$\vd'$) is a
	\da, then we call $(V,\ell_{\dva},r_{\dva},\ell_{\vda},r_{\vda})$ a  {\bf representation} of $(A,\dva,\vda)$. In this case,   ($A\oplus V$,$\dv'$,$\vd'$) is called the {\bf semi-direct product diassociative algebra} of $(A,\dva, \vda)$, denoted by $ A\ltimes^{\ell_{\dva},r_{\dva}}_{\ell_{\vda},r_{\vda}} V $.
\end{defn}

\begin{prop}\mlabel{pro:repr}
	Let $(A,\dva,\vda)$ be a \da and let $V$ be a vector space. Suppose that  $\ell_{\dva},r_{\dva},\ell_{\vda},r_{\vda}: A\to \End(V)$ are linear maps. Then $(V,\ell_{\dva},r_{\dva},\ell_{\vda},r_{\vda})$ is a representation of $(A,\dva,\vda)$ if and only if the following equations hold.
	\begin{align}
		&\ell_{\dva}(x)\ell_{\dva}(y)=\ell_{\dva}(x)\ell_{\vda}(y)=\ell_{\dva}(x\dva y),\mlabel{eq:repr1}\\
		&\ell_{\dva}(x)r_{\dva}(y)=\ell_{\dva}(x)r_{\vda}(y)=r_{\dva}(y)\ell_{\dva}(x),\mlabel{eq:repr2}\\
		&r_{\dva}(x\dva y)=r_{\dva}(x\vda y)=r_{\dva}(y)r_{\dva}(x),\mlabel{eq:repr3}\\
		&\ell_{\dva}(x\vda y)=\ell_{\vda}(x)\ell_{\dva}(y),\mlabel{eq:repr4}\\
		&r_{\dva}(x)\ell_{\vda} (y)=\ell_{\vda}(y)r_{\dva}(x), \mlabel{eq:repr5}\\
		&r_{\dva}(x)r_{\vda} (y)=r_{\vda}(y\dva x), \mlabel{eq:repr6}\\
		&\ell_{\vda}(x\dva y)=\ell_{\vda}(x\vda y)=\ell_{\vda}(x)\ell_{\vda}(y),\mlabel{eq:repr7}\\
		&r_{\vda}(x)\ell_{\dva} (y)=	r_{\vda}(x)\ell_{\vda} (y)=\ell_{\vda}(y)r_{\vda}(x), \mlabel{eq:repr8}\\
		&r_{\vda}(x)r_{\dva} (y)=	r_{\vda}(x)r_{\vda} (y)=r_{\vda}(y\vda x), 	\quad x, y\in A. \mlabel{eq:repr9}
	\end{align}
\end{prop}
\begin{proof}
This follows from a direct computation.
\end{proof}

 \begin{prop}\mlabel{prop:dual-rep}
 Let $(A,\dva,\vda)$ be a \da. If $(V,\ell_{\dva},r_{\dva},\ell_{\vda},r_{\vda})$  is a representation of $(A,\dva,\vda)$, then $(V^*,-r^*_{\oba},\ell^*_{\vda},r^*_{\dva},\ell^*_{\oba})$ is also a representation of $A$, where $r^*_{\oba}:=r^*_{\vda}-r^*_{\dva}$ and $\ell^*_{\oba}:=\ell^*_{\vda}-\ell^*_{\dva}$.
 \end{prop}
 \begin{proof}
 Suppose that $(V,\ell_{\dva},r_{\dva},\ell_{\vda},r_{\vda})$  is a representation of $(A,\dva,\vda)$. By Proposition~\mref{pro:repr},  we shall verify that $(V^*,-r^*_{\oba},\ell^*_{\vda},r^*_{\dva},\ell^*_{\oba})$ satisfies  \meqref{eq:repr1}-\meqref{eq:repr9}. We only prove ~\meqref{eq:repr1}, since the remaining cases follow from  an analogous proof. For all $x,y\in A$, $u\in V$ and $a^*\in V^*$, we have
 	\begin{align*}
 	\langle r^*_{\oba}(x)r^*_{\oba}(y)a^*,u \rangle
&=\langle a^*,r_{\oba}(y)r_{\oba}(x)u\rangle\\
&=\langle a^*,(r_{\dva}(y)r_{\dva}(x)-r_{\vda}(y)r_{\dva}(x)-r_{\dva}(y)r_{\vda}(x)
+r_{\vda}(y)r_{\vda}(x))u\rangle\\
&\stackrel{\eqref{eq:repr9}}{=}\langle a^*,(r_{\dva}(y)r_{\dva}(x)-r_{\dva}(y)r_{\vda}(x))u\rangle	
=\langle -r^*_{\oba}(x)r^*_{\dva}(y)a^*,u\rangle\\
&\stackrel{\eqref{eq:repr3}\eqref{eq:repr6}}{=}\langle a^*,
(r_{\dva}(x{\dva} y)-r_{\vda}(x{\dva} y))u\rangle\\
&=\langle(-r^*_{\oba})(x{\dva} y)a^*,u \rangle.
\end{align*}
Thus, $(-r^*_{\oba})(x)(-r^*_{\oba})(y)=(-r^*_{\oba})(x)r^*_{\dva}(y)=(-r^*_{\oba})(x{\dva} y)$. 
\end{proof}
Building on the conclusions presented above, the following results can be derived.
\begin{exam}\mlabel{ex:dual}
	Let $(A,{\dva},\vda)$ be a \da. Then
	\begin{enumerate}
		\item $(A, L_{\dva},R_{\dva},L_{\vda},R_{\vda})$ is a representation of $(A,{\dva},\vda)$, called the {\bf adjoint representation} of $(A,\dva,\vda)$.
		\item $(A^*, -R^*_{\oba},L^*_{\vda},R^*_{\dva},L^*_{\oba})$ is also a representation of $(A,\dva,\vda)$, called the {\bf dual representation} of $(A, L_{\dva},R_{\dva},L_{\vda},R_{\vda})$.\mlabel{it:dual2}
	\end{enumerate}
\end{exam}

\begin{defn}
	Let $(A,\dva,\vda)$ be a \da. Let $(V,\ell_{\dva},r_{\dva},\ell_{\vda},r_{\vda})$ and $(V',\ell'_{\dva},r'_{\dva}$,
	$\ell'_{\vda},r'_{\vda})$  be two representations of $(A,\dva,\vda)$.
If there exists a linear isomorphism $\phi:V\rightarrow V'$  such that
	\begin{align*}
	 \phi \ell_\ast(x) (u )=\ell'_\ast(x)\phi(u ),\quad   \phi r_\ast(x) (u )=r'_\ast(x)\phi(u ),\quad\ast\in\{\dva,\vda\}, x\in A, u \in V ,
	\end{align*}
then we say that $(V,\ell_{\dva},r_{\dva},\ell_{\vda},r_{\vda})$ and $(V',\ell'_{\dva},r'_{\dva},\ell'_{\vda},r'_{\vda})$ are {\bf equivalent}.
\end{defn}
\begin{coro}
	Let $(A,\dva,\vda)$ be a quadratic \da. Then the adjoint representation $(A, L_{\dva},R_{\dva},
L_{\vda},R_{\vda})$  is equivalent to  its dual representation $(A^*, -R^*_{\oba},L^*_{\vda},R^*_{\dva},L^*_{\oba})$.
\end{coro}
\begin{proof}
Suppose that $\omega$ is a nondegenerate antisymmetric bilinear form on $A$. Define a linear isomorphism $\phi: A\to A^*$ by defining $\langle \phi(x), y\rangle:=\omega(x,y)$ for all $x,y\in A$. By ~\meqref{eq:omega}, we obtain
$$\langle \phi(L_{\dva}(x) y),z\rangle=\omega(x\dva y,z)=-\omega(y,z\oba x)=-\langle \phi(y), R_{\oba}(x)z\rangle=\langle-R_{\oba}^*(x)\phi(y),z\rangle,$$
proving $\phi(L_{\dva}(x)y)=-R^*_{\oba}(x)\phi(y)$. Similarly, 
$$\langle \phi(R_{\dva}(x) y),z\rangle=\omega(y\dva x,z)\stackrel{\meqref{eq:omega1}}{=}\omega(y,x\vda z)=\langle L^*_{\vda}(x)\phi(y),z\rangle.$$
$$\langle \phi(L_{\vda}(x) y),z\rangle=\omega(x\vda y,z)\stackrel{\meqref{eq:omega1}}{=}\omega(y,z\dva x)=\langle R_{\dva}^*(x)\phi(y),z\rangle.$$
$$\langle \phi(R_{\vda}(x) y),z\rangle=\omega(y\vda x,z)\stackrel{\meqref{eq:omega}}{=}\omega(y,x\oba  z)=\langle L^*_{\oba}(x)\phi(y),z\rangle.$$
Thus, $\phi(R_{\dva}(x)y)=L^*_{\vda}(x)\phi(y),\,\phi(L_{\vda}(x) y)=R_{\dva}^*(x)\phi(y),\,\phi(R_{\vda}(x) y)= L^*_{\oba}(x)\phi(y)$.
\end{proof}

\section{Matched pairs, Manin triples of diassociative algebras and diassociative bialgebras}
\mlabel{sec:3}
In this section,  we introduce the notions of  matched pairs and  Manin triples of \das. Based on a particular matched pair of \das, we introduce the notion of a \dabi.  We then establish the equivalence among such matched pairs, Manin triples of \das and \dabis.

\subsection{Matched pairs  and Manin triples of diassociative  algebras}

\begin{defn}	\mlabel{defn:matched}
	Let $(A,\dva,\vda)$ and $(B,\dashv_B,\vdb)$ be \das.
	Let $\ell_{\dva},r_{\dva},\ell_{\vda},r_{\vda}:A\to \End(B)$ and $\ell_{\dashv_B},r_{\dashv_B},\ell_{\vdb},r_{\vdb}: B\to \End(A)$ be linear maps. Define two binary operations $\dashv_{A\bowtie B},\vd_{A\bowtie B}$ on the direct sum $A\oplus B$ by
	\begin{eqnarray}
		(x+a)\dashv_{A\bowtie B} (y+b):=(x\dva y+r_{\dashv_B}(b)x+\ell_{\dashv_B}(a)y)+(a\dashv_B b+\ell_{\dva}(x)b+r_{\dva}(y)a),\mlabel{eq:m1}\\
		(x+a)\vd_{A\bowtie B} (y+b):=(x\vda y+r_{\vdb}(b)x+\ell_{\vdb}(a)y)+(a\vdb b+\ell_{\vda}(x)b+r_{\vda}(y)a),\mlabel{eq:m2}
	\end{eqnarray}
	for all $x,y\in A$ and $a,b\in B$. If $(A\oplus B, \dashv_{A\bowtie B},\vd_{A\bowtie B})$ is a \da, denoted by $A\bowtie_{\ell_{\dashv_B},r_{\dashv_B},\ell_{\vdb},r_{\vdb}}^{\ell_{\dva},r_{\dva},\ell_{\vda},r_{\vda}} B$ or simply $A\bowtie B$,
	then the $10$-tuple
	$(A,B,\ell_{\dva},r_{\dva},\ell_{\vda},r_{\vda},\ell_{\dashv_B},r_{\dashv_B},\ell_{\vdb},r_{\vdb})$
	is called a {\bf matched pair of diassociative algebras}.
\end{defn}

The structure of matched pairs of \das can be characterized through their representations and several compatibility conditions.
\begin{lemma}
	Let $(A,\dva,\vda)$ and $(B,\dashv_B,\vdb)$ be  diassociative algebras. Let $\ell_{\dva},r_{\dva},\ell_{\vda},r_{\vda}:A\rightarrow \End(B)$ and $\ell_{\dashv_B},r_{\dashv_B},\ell_{\vdb},r_{\vdb}:B\rightarrow \End(A)$ be linear maps. Then $(A,B,\ell_{\dva},r_{\dva},\ell_{\vda},r_{\vda},\ell_{\dashv_B},r_{\dashv_B},\ell_{\vdb}$,
	$r_{\vdb})$ is a matched pair of \das if and only if
	\begin{enumerate}
		\item $(B,\ell_{\dva},r_{\dva},\ell_{\vda},r_{\vda})$ is a representation of $(A,\dva,\vda)$,
		\item $(A,\ell_{\dashv_B},r_{\dashv_B},\ell_{\vdb},r_{\vdb})$ is a representation of $(B,\dashv_B,\vdb)$, and
		\item the following compatibility conditions hold: for all $x, y \in A$ and $a, b \in B$,
\begin{small}
		\begin{flalign}
			\ell_{\dva}(x)(a\dashv_B b)
			&=\ell_{\dva}(x)(a\vdb b)=(\ell_{\dva}(x)a)\dashv_B b+\ell_{\dva}(r_{\dashv_B}(a)x)b,\mlabel{eq:comp1}\\
			\ell_{\dashv_B}(a)(x\dva y)
				&=\ell_{\dashv_B}(a)(x\vda y)=(\ell_{\dashv_B}(a)x)\dva y+\ell_{\dashv_B}(r_{\dva}(x)a)y,\mlabel{eq:comp2}\\
			x\dva(r_{\dashv_B}(a)y)+r_{\dashv_B}(\ell_{\dva}(y)a)x
				&=x\dva (r_{\vdb}(a)y)+r_{\dashv_B}(\ell_{\vda}(y)a)x=r_{\dashv_B}(a)(x \dva y),\mlabel{eq:comp3}\\
			a\dashv_B(r_{\dva}(x)b)+r_{\dva}(\ell_{\dashv_B}(b)x)a	&=a\dashv_B (r_{\vda}(x)b)+r_{\dva}(\ell_{\vdb}(b)x)a=r_{\dva}(x)(a \dashv_B b),\mlabel{eq:comp4}\\
			x\dva(\ell_{\dashv_B}(a)y)+r_{\dashv_B}(r_{\dva}(y)a)x	&=x\dva (\ell_{\vdb}(a)y)+r_{\dashv_B}(r_{\vda}(y)a)x
			=(r_{\dashv_B}(a)x) \dva y+	\ell_{\dashv_B}(\ell_{\dva}(x)a)y,\mlabel{eq:comp5}\\
			a\dashv_B(\ell_{\dva}(x)b)+r_{\dva}(r_{\dashv_B}(b)x)a	&=a\dashv_B (\ell_{\vda}(x)b)+r_{\dva}(r_{\vdb}(b)x)a=(r_{\dva}(x)a) \dashv_B b+\ell_{\dva}(\ell_{\dashv_B}(a)x)b,\mlabel{eq:comp6}\\
			\ell_{\vdb}(a)(x\dva y)	&=(\ell_{\vdb}(a)x)\dva y+\ell_{\dashv_B}(r_{\vda}(x)a)y,\mlabel{eq:comp7}\\
			\ell_{\vda}(x)(a\dashv_B b)	&=(\ell_{\vda}(x)a)\dashv_B b+\ell_{\dva}(r_{\vdb}(a)x)b,\mlabel{eq:comp8}\\
			r_{\dashv_B}(a)(x\vda y)	&=x\vda(r_{\dashv_B}(a)y)+r_{\vdb}(\ell_{\dva}(y)a)x,\mlabel{eq:comp9}\\
			r_{\dva}(x)(a\vdb b)	&=a\vdb(r_{\dva}(x)b)+r_{\vda}(\ell_{\dashv_B}(b)x)a,\mlabel{eq:comp10}\\
			(r_{\vdb}(a)x)\dva y+\ell_{\dashv_B}(\ell_{\vda}(x)a)y	&=x\vda(\ell_{\dashv_B}(a)y)+r_{\vdb}(r_{\dva}(y)a)x,\mlabel{eq:comp11}\\
			(r_{\vda}(x)a)\dashv_B b+\ell_{\dva}(\ell_{\vdb}(a)x)b	&=a\vdb(\ell_{\dva}(x)b)+r_{\vda}(r_{\dashv_B}(b)x)a,\mlabel{eq:comp12}\\	
			r_{\vdb}(a)(x\dva y)	&=r_{\vdb}(a)(x\vda y)=r_{\vdb}(\ell_{\vda}(y)a)x+x\vda(r_{\vdb}(a)y),\mlabel{eq:comp13}\\
			r_{\vda}(x)(a\dashv_B b)	&=r_{\vda}(x)(a\vdb b)=r_{\vda}(\ell_{\vdb}(b)x)a+a\vdb(r_{\vda}(x)b),\mlabel{eq:comp14}\\
			(\ell_{\dashv_B}(a)x)\vda y+\ell_{\vdb}(r_{\dva}(x)a)y	&=(\ell_{\vdb}(a)x)\vda y+\ell_{\vdb}(r_{\vda}(x)a)y=\ell_{\vdb}(a)(x\vda y),\mlabel{eq:comp15}\\
			(\ell_{\dva}(x)a)\vdb b+\ell_{\vda}(r_{\dashv_B}(a)x)b	&=(\ell_{\vda}(x)a)\vdb b+\ell_{\vda}(r_{\vdb}(a)x)b=\ell_{\vda}(x)(a\vdb b),\mlabel{eq:comp16}\\
			(r_{\dashv_B}(a)x)\vda y+\ell_{\vdb}(\ell_{\dva}(x)a)y	&=(r_{\vdb}(a)x)\vda y+\ell_{\vdb}(\ell_{\vda}(x)a)y
			=x\vda(\ell_{\vdb}(a)y)+r_{\vdb}(r_{\vda}(y)a)x,\mlabel{eq:comp17}\\
			(r_{\dva}(x)a)\vdb b+\ell_{\vda}(\ell_{\dashv_B}(a)x)b	&=(r_{\vda}(x)a)\vdb b+\ell_{\vda}(\ell_{\vdb}(a)x)b
			=a\vdb(\ell_{\vda}(x)b)+r_{\vda}(r_{\vdb}(b)x)a.\mlabel{eq:comp18}
		\end{flalign}
\end{small}
	\end{enumerate}
\mlabel{lem:match}
\end{lemma}
\begin{proof}
This follows from a direct computation.
\end{proof}
\begin{defn}\mlabel{defn:dia}
	A {\bf Manin triple of diassociative algebras} is a collection $((A\oplus B,\dv_d,\vd_d,\omega ),(A,\dva,\vda),(B,\dvb,\vdb))$ such that $(A\oplus B,\dv_d,\vd_d,\omega)$ is a quadratic diassociative algebra, $(A,\dva,\vda),(B,\dvb,\vdb)$ are diassociative subalgebras and $A,B$ are isotropic with respect to $\omega $, that is, 
	\begin{equation*}
		\omega (x,y)=0=\omega (a,b),\; x,y\in A, a,b\in B.
	\end{equation*}
	A   Manin triple of diassociative algebras is called {\bf standard} if $B=A^{*}$ and $\omega=\omega_d$ is given by
	\begin{equation}\mlabel{eq:omd}
		\omega_d(x+a^*,y+b^*)=\langle a^*,y\rangle-\langle x,b^*\rangle,\; x,y\in A, a^{*},b^{*}\in A^{*}.
	\end{equation}
\end{defn}

As in \cite{Bai}, every Manin triple of diassociative algebras is isomorphic to a standard one.
The following result establishes the fundamental equivalence between a certain matched pair of \das and a Manin triple of \das.
\begin{theorem}\mlabel{thm:Manin}
	Let $(A,\dva,\vda)$ and $(A^*,\dvs,\vds)$ be \das. Then $(A,A^*,-R^*_{\oba}, L^*_{\vda},\\
R^*_{\dva},L^*_{\oba}$, $-R^*_{\obas},L^*_{\vds},R^*_{\dvs},L^*_{\obas})$ is a matched pair of \das
	if and only if  $((A\oplus A^*,\dashv_d,\vd_d,\omega_d),(A,\dva,\vda), (A^*,\dvs,\vds))$ is a standard Manin triple of \das.
\end{theorem}
\begin{proof}
	($\Rightarrow$)
	Suppose that $(A,A^*,-R^*_{\oba},L^*_{\vda},R^*_{\dva},L^*_{\oba},-R^*_{\obas},L^*_{\vds},R^*_{\dvs},L^*_{\obas})$ is a matched pair of \das.
	Then by Definition~\mref{defn:matched}, $(A\oplus A^*,\dashv_{A \bowtie A^*},\vd_{A \bowtie A^*})$ is a \da. It suffices to verify that $\omega_d$, as defined by ~\meqref{eq:omd}, is invariant, i.e., that it  satisfies ~\meqref{eq:omega}. For all $x,y,z\in A$ and $a^*,b^*, c^*\in A^*$, we have
	\begin{eqnarray*}
		&&\omega_d(x+a^*,  (y+b^*) \dashv_{A \bowtie A^*} (z+c^*))\\
&=&\omega_d (x+a^*,y\dva z+L^*_{\vds}(c^*)y-R^*_{\obas}(b^*)z+b^*\dvs c^*-R^*_{\oba}(y)c^*+L^*_{\vda}(z)(b^*))\\
		&=&\langle a^*,y\dva z+L^*_{\vds}(c^*)y-R^*_{\obas}(b^*)z \rangle-\langle x,b^*\dvs c^*-R^*_{\oba}(y)c^*+L^*_{\vda}(z)(b^*)\rangle\\
&=&\langle a^*,y\dva z \rangle+ \langle c^*\vds a^* ,y\rangle-\langle a^*\obas b^*,z \rangle
-\langle x,b^*\dvs c^* \rangle+\langle x \oba y,  c^*\rangle-\langle z\vda x,b^*\rangle.\\
		&&\omega_d(z+c^*, (x+a^*) \ob_{A \bowtie A^*} (y+b^*))\\
&=&\omega_d (z+c^*,x\oba y -L^*_{\dvs}(b^*)x+R^*_{\vds}(a^*)y+a^*\obas b^*+R^*_{\vda}(x)b^*-L^*_{\dva}(y)a^*)\\
&=&\langle c^*,x\oba y -L^*_{\dvs}(b^*)x+R^*_{\vds}(a^*)y \rangle-\langle z,a^*\obas b^*+R^*_{\vda}(x)b^*-L^*_{\dva}(y)a^* \rangle\\
&=&\langle c^*,x\oba y \rangle-\langle b^*\dvs c^*,x\rangle+\langle c^*\vds a^*,y \rangle-\langle z, a^*\obas b^*\rangle-\langle z\vda x,b^* \rangle+\langle y\dva z,a^*\rangle.
	\end{eqnarray*}
Thus, $$\omega_d(x+a^*,  (y+b^*) \dashv_{A \bowtie A^*} (z+c^*))=\omega_d((z+c^*), (x+a^*) \ob_{A \bowtie A^*} (y+b^*)).$$
On the other hand,
	\begin{eqnarray*}
		&&\omega((x+a^*)\vd_{A \bowtie A^*} (y+b^*), z+c^*)\\&=&\omega (x\vda y+L^*_{\obas}(b^*)x+R^*_{\dvs}(a^*)y+a^*\vds b^*+R^*_{\dva}(x)b^*+L^*_{\oba}(y)(a^*),z+c^*)\\
&=&\langle a^*\vds b^*,z \rangle+ \langle b^*,z\dva x\rangle+\langle a^*,y\oba z \rangle-\langle x\vda y,c^* \rangle-\langle x, b^*\obas c^*\rangle-\langle y,c^*\dvs a^*\rangle.\\
		&&\omega(x+a^*, (y+b^*)\ob_{A \bowtie A^*}(z+c^*))\\
&=&\omega (x+a^*,y\oba z-L^*_{\dvs}(c^*)y+R^*_{\vds}(b^*)z+b^*\obas c^*+R^*_{\vda}(y)c^*-L^*_{\dva}(z)b^*)\\
&=&\langle a^*,y\oba z \rangle-\langle c^*\dvs a^*,y\rangle+\langle a^*\vds b^*,z \rangle-
		\langle x, b^*\obas c^*\rangle-\langle x\vda y,c^* \rangle+ \langle z\dva x,b^*\rangle.
	\end{eqnarray*}
This leads to
$$\omega((x+a^*)\vd_{A \bowtie A^*} (y+b^*), z+c^*)=\omega(x+a^*, (y+b^*)\ob_{A \bowtie A^*}( z+c^*)).$$
\smallskip
\noindent	
($\Leftarrow$) Suppose that $((A\oplus A^*,\dashv_d,\vd_d,\omega_d),(A,\dva,\vda), (A^*,\dvs,\vds))$ is a standard Manin triple of diassociative algebras. By Definition \mref{defn:matched},  it suffices to prove that $\dashv_d$ and $\vd_d$ satisfy ~\meqref{eq:m1}--\meqref{eq:m2}. For all $x,y,z\in A$ and all $a^*, b^*,c^*\in A^*$,
	\begin{eqnarray*}
&\omega_d( x\dashv_d a^*,y)=-\omega_d( y,x\dashv_d a^*)\stackrel{\eqref{eq:omega}}{=}-\omega_d(a^*, y\oba x)
=\langle -R^*_{\oba}(x)a^*,y \rangle.\\
&\omega_d( x\dashv_d a^*,b^*)\stackrel{\eqref{eq:omega1}}{=}\omega_d(x,a^*\vds b^*)=\omega_d(x,a^*\vds b^*)
=-\langle L^*_{\vds}(a^*)x,b^* \rangle.
	\end{eqnarray*}
This gives $$\omega_d( x\dashv_d a^*,y+b^*)=\langle -R^*_{\oba}(x)a^*,y \rangle-\langle L^*_{\vds}(a^*)x,b^* \rangle=\omega_d(L^*_{\vds}(a^*)x-R^*_{\oba}(x)a^*,y+b^*).$$
Thus,  $x\dashv_d a^*=L^*_{\vds}(a^*)x-R^*_{\oba}(x)a^*$.  By a similar argument, we obtain
	\begin{eqnarray*}
&&\omega_d( a^*\dashv_d x,y)=\omega_d(a^*,x\vd_d y)=\langle L^*_{\vda}(x)a^*,y \rangle=
\omega_d( L^*_{\vda}(x)a^* x,y)
,\\
&&
\omega_d( a^*\dashv_d x,b^*)=-\omega_d(x, b^*\obas a^*)=\langle R^*_{\obas}(a^*)x, b^* \rangle=-\omega_d( R^*_{\obas}(a^*)x, b^*).
	\end{eqnarray*}
So $a^*\dashv_d x=-R^*_{\obas}(a^*)x+L^*_{\vda}(x)a^*$, and hence $\dashv_d$ satisfies ~\eqref{eq:m1} with $\ell_{\dva}=-R^*_{\oba}, r_{\dva}=L^*_{\vda},\ell_{\dashv_B}=-R^*_{\obas},r_{\dashv_B}=L^*_{\vds}$.
Similarly, we can verify that $\vd_d$ satisfies ~\eqref{eq:m2} with $\ell_{\vda}=R^*_{\dva}, r_{\vda}=L^*_{\oba},\ell_{\vdb}=R^*_{\dvs},r_{\vdb}=L^*_{\obas}$.
\end{proof}

Building upon the preceding preparatory steps, we now arrive at the central concept of a diassociative bialgebra.
\begin{defn}
	A {\bf diassociative bialgebra} is a quintuple $(A,\dva,\vda,\Delta_{\dva},\Delta_{\vda})$ consisting of a vector space $A$  together with linear maps
	$\dva,\vda : A\ot A \to A ,\, \Delta_{\dva},\Delta_{\vda}: A \to  A\ot A $
	such that
	\begin{enumerate}
		\item $(A,\dva,\vda)$ is a \da,
		\item $(A,\Delta_{\dva},\Delta_{\vda})$ is a \dca, and
		\item the following compatibility conditions hold: for all $x,y \in A$,
\begin{small}
		\begin{flalign}
			&\Delta_{\vda}(x\dva y)=(\id \otimes L_{\dva} (x))\Delta_{\vda}(y)-(R_{\oba}(y)\otimes\id )\Delta_{\vda}(x)
			=(\id \otimes L_{\dva}(x))\Delta_{\oba}(y)+(R_{\dva}(y)\otimes \id )\Delta_{\vda}(x),\mlabel{eq:bi2}\\
			&\sigma((L_{\vda}(y)\otimes\id)\Delta_{\vda}(x))-(L_{\dva}(x)\otimes \id)\Delta_{\oba}(y)
			=(L_{\dva}(x)\otimes \id)\Delta_{\dva}(y)+\sigma((L_{\oba}(y)\otimes\id)\Delta_{\vda}(x))\mlabel{eq:bi3}\\
			&\quad\quad\quad\quad\quad\quad\quad\quad\quad\quad\quad\quad\quad\quad\quad\quad\,=\sigma((\id \otimes R_{\dva}(y))\Delta_{\vda}(x))+(\id \otimes R_{\oba}(x))\Delta_{\oba}(y),\mlabel{eq:bi4}\\
			&\sigma((\id \otimes R_{\vda}(y))\Delta_{\vda}(x))-(\id \otimes R_{\oba}(x) )\Delta_{\dva}(y)
			=(\id \otimes R_{\dva}(x))\Delta_{\dva}(y)+\sigma((\id \otimes R_{\vda}(y))\Delta_{\oba}(x))\mlabel{eq:bi5}\\
			&\quad\quad\quad\quad\quad\quad\quad\quad\quad\quad\quad\quad\quad\quad\quad\quad\,=(L_{\vda}(x)\otimes\id)\Delta_{\dva}(y)+\sigma((L_{\oba}(y)\otimes\id)\Delta_{\oba}(x)),\mlabel{eq:bi6}\\
			&\Delta_{\dva}(x\vda y)=(\id \otimes L_{\vda} (x))\Delta_{\dva}(y)-(R_{\vda}(y)\otimes\id )\Delta_{\oba}(x)=(\id \otimes L_{\oba}(x))\Delta_{\dva}(y)+(R_{\vda}(y)\otimes \id )\Delta_{\dva}(x),\mlabel{eq:bi8}\\
			&\Delta_{\dva}(x\dva y)=(R_{\dva}(y)\otimes \id )\Delta_{\dva}(x)-(\id \otimes L_{\oba}(x))\Delta_{\oba}(y),\mlabel{eq:bi9}\\
			&\Delta_{\vda}(x\vda y)=(\id \otimes L_{\vda}(x))\Delta_{\vda}(y)-(R_{\oba}(y)\otimes \id )\Delta_{\oba}(x),\mlabel{eq:bi10}
\end{flalign}
where $\Delta_{\oba}:=\Delta_{\vda}-\Delta_{\dva}$.
\end{small}

	\end{enumerate}
\end{defn}
The following result shows that diassociative bialgebras precisely correspond to a certain matched pair of \das.
\begin{theorem}\mlabel{thm:matched}
	Let $(A,\dva,\vda)$ and $(A^*,\dvs,\vds)$ be \das. Let $\Delta_{\dva},\Delta_{\vda}$ be the linear dual  of $\dvs$ and $\vds$, respectively. Then 
$(A,A^*,-R^*_{\oba},L^*_{\vda},R^*_{\dva},L^*_{\oba},-R^*_{\obas},L^*_{\vds},R^*_{\dvs},L^*_{\obas})$ is a matched pair of \das
if and only if $(A,\dva,\vda,\Delta_{\dva},\Delta_{\vda})$ is a \dabi.
\end{theorem}
\begin{proof} By Example~\mref{ex:dual}\meqref{it:dual2}, $(A^*, -R^*_{\oba},L^*_{\vda}$,
	$R^*_{\dva},L^*_{\oba})$ is  a representation of $(A,\dva,\vda)$, and  $(A$, 
	$-R^*_{\obas}, 
L^*_{\vds},R^*_{\dvs},L^*_{\obas})$ is a representation of $(A^*,\dvs,\vds)$.
Then  by Lemma~\mref{lem:match}, $(A,A^*,-R^*_{\oba},L^*_{\vda}$,
$R^*_{\dva},L^*_{\oba}, 
-R^*_{\obas},L^*_{\vds},R^*_{\dvs},L^*_{\obas})$ is a matched pair of \das if and only if it satisfies \meqref{eq:comp1}--\meqref{eq:comp18}, where the maps are assigned as follows: 
$$\ell_{\dva}=-R^*_{\oba},\,r_{\dva}=L^*_{\vda},\,\ell_{\vda}=R^*_{\dva},\,r_{\vda}=L^*_{\oba},\,\ell_{\dvb}=-R^*_{\obas},\,r_{\dvb}=L^*_{\vds},\ell_{\vdb}=R^*_{\dvs},r_{\vdb}=L^*_{\obas}.$$
That is, for all $x,y \in A$, $a^*,b^* \in A^*$,
\begin{small}
	\begin{flalign}	
&	R^*_{\oba}(x)(a^*\dvs b^*)
		=R^*_{\oba}(x)(a^*\vds b^*)
=(R^*_{\oba}(x)a^*)\dvs b^*+R^*_{\oba}(L^*_{\vds}(a^*)x)b^*,\mlabel{eq:ppd1}\\
		&R^*_{\obas}(a^*)(x\dva y)
	=R^*_{\obas}(a^*)(x\vda y)
=(R^*_{\obas}(a^*)x)\dva y+R^*_{\obas}(L^*_{\vda}(x)a^*)y,\mlabel{eq:ppd2}\\
		&x\dva(L^*_{\vds}(a^*)y)-L^*_{\vds}(R^*_{\oba}(y)a^*)x
		=x\dva L^*_{\obas}(a^*)y+L^*_{\vds}(R^*_{\dva}(y)a^*)x
=L^*_{\vds}(a^*)(x \dva y),\mlabel{eq:ppd3}\\
	&a^*\dvs(L^*_{\vda}(x)b^*)-L^*_{\vda}(R^*_{\obas}(b^*)x)a^*	=a^*\dvs (L^*_{\oba}(x)b^*)+L^*_{\vda}(R^*_{\dvs}(b^*)x)a^*
=L^*_{\vda}(x)(a^* \dvs b^*),\mlabel{eq:ppd4}\\
		&L^*_{\vds}(L^*_{\vda}(y)a^*)x-x\dva(R^*_{\obas}(a^*)y)	=x\dva (R^*_{\dvs}(a^*)y)+L^*_{\vds}(L^*_{\oba}(y)a^*)x\nonumber\\
&\qquad\quad\quad\qquad\qquad\qquad\quad\quad\qquad	=(L^*_{\vds}(a^*)x) \dva y+	R^*_{\obas}(R^*_{\oba}(x)a^*)y,\mlabel{eq:ppd5}\\
		&L^*_{\vda}(L^*_{\vds}(b^*)x)a^*-a^*\dvs(R^*_{\oba}(x)b^*)	=a^*\dvs (R^*_{\dva}(x)b^*)+L^*_{\vda}(L^*_{\obas}(b^*)x)a^*\nonumber\\
&\quad\quad\quad\quad\quad\quad\quad\quad\qquad\quad\qquad\qquad=(L^*_{\vda}(x)a^*) \dvs b^*+(R^*_{\oba}(R^*_{\obas}(a^*)x)b^*,\mlabel{eq:ppd6}\\
		&R^*_{\dvs}(a^*)(x\dva y)	=(R^*_{\dvs}(a^*)x)\dva y-R^*_{\obas}(L^*_{\oba}(x)a^*)y,\mlabel{eq:ppd7}\\
		&R^*_{\dva}(x)(a^*\dvs b^*)	=(R^*_{\dva}(x)a^*)\dvs b^*
-R^*_{\oba}(L^*_{\obas}(a^*)x)b^*,\mlabel{eq:ppd8}\\
		&L^*_{\vds}(a^*)(x\vda y)=x\vda(L^*_{\vds}(a^*)y)-L^*_{\obas}(R^*_{\oba}(y)a^*)x,\mlabel{eq:ppd9}\\
		&L^*_{\vda}(x)(a^*\vds b^*)	=a^*\vds(L^*_{\vda}(x)b^*)-L^*_{\oba}(R^*_{\obas}(b^*)x)a^*,\mlabel{eq:ppd10}\\
		&(L^*_{\obas}(a^*)x)\dva y-R^*_{\obas}(R^*_{\dva}(x)a^*)y	=L^*_{\obas}(L^*_{\vda}(y)a^*)x-x\vda(R^*_{\obas}(a^*)y),\mlabel{eq:ppd11}\\
		&(L^*_{\oba}(x)a^*)\dvs b^*-R^*_{\oba}(R^*_{\dvs}(a^*)x)b^*	=L^*_{\oba}(L^*_{\vds}(b^*)x)a^*-a^*\vds (R^*_{\oba})(x)b^*),\mlabel{eq:ppd12}\\	
		&L^*_{\obas}(a^*)(x\dva y)	=L^*_{\obas}(a^*)(x\vda y)=L^*_{\obas}(R^*_{\dva}(y)a^*)x+x\vda(L^*_{\obas}(a^*)y),\mlabel{eq:ppd13}\\
		&L^*_{\oba}(x)(a^*\dvs b^*)	=L^*_{\oba}(x)(a^*\vds b^*)=L^*_{\oba}(R^*_{\dvs}(b^*)x)a^*+a^*\vds(L^*_{\oba}(x)b^*),\mlabel{eq:ppd14}\\
		&R^*_{\dvs}(L^*_{\vda}(x)a^*)y-R^*_{\obas}(a^*)x\vda y	=(R^*_{\dvs}(a^*)x)\vda y+R^*_{\dvs}(L^*_{\oba}(x)a^*)y=R^*_{\dvs}(a^*)(x\vda y),\mlabel{eq:ppd15}\\
		&R^*_{\dva}(L^*_{\vds}(a^*)x)b^*-R^*_{\oba}(x)a^*\vds b^*	=(R^*_{\dva}(x)a^*)\vds b^*+R^*_{\dva} (L^*_{\obas}(a^*)x)b^*=R^*_{\dva}(x)(a^*\vds b^*),\mlabel{eq:ppd16}\\
		&(L^*_{\vds}(a^*)x)\vda y-R^*_{\dvs}(R^*_{\oba}(x)a^*)y	=(L^*_{\obas}(a^*)x)\vda y+R^*_{\dvs}(R^*_{\dva}(x)a^*)y
\nonumber\\
&\quad\quad\quad\qquad\qquad\qquad\quad\qquad\qquad=x\vda(R^*_{\dvs}(a^*)y)+L^*_{\obas}(L^*_{\oba}(y)a^*)x,\mlabel{eq:ppd17}\\
		&(L^*_{\vda}(x)a^*)\vds b^*-R^*_{\dva}(R^*_{\obas}(a^*)x)b^*	=(L^*_{\oba}(x)a^*)\vds b^*+R^*_{\dva}(R^*_{\dvs}(a^*)x)b^*\nonumber\\
&\qquad\quad\quad\qquad\qquad\qquad\quad\qquad\qquad
=a^*\vds (R^*_{\dva}(x)b^*)+L^*_{\oba}(L^*_{\obas}(b^*)x)a^*. \mlabel{eq:ppd18}
	\end{flalign}
\end{small}
Moreover, from \meqref{eq:ppd1} we deduce
\begin{eqnarray*}	
	\langle 	R^*_{\oba}(x)(a^*\dvs b^*),y	\rangle& =&	\langle a^*\dvs b^*, y\oba x  	\rangle  =\langle 	L^*_{\oba}(y)(a^*\dvs b^*),x \rangle=\langle  \Delta_{\dva}(y\oba x ),a^*\ot b^*\rangle.\\
	\langle 	R^*_{\oba}(x)(a^*\vds b^*),y	\rangle&=&\langle a^*\vds b^*, y\oba x  	\rangle  =\langle L^*_{\oba}(y)(a^*\vds b^*),x\rangle=\langle  \Delta_{\vda}(y\oba x),a^*\ot b^*\rangle.\\
\langle(R^*_{\oba}(x)a^*)\dvs b^*,y\rangle&=&\langle R^*_{\oba}(x)a^*,R^*_{\dvs}(b^*)y\rangle=\langle a^*,(R^*_{\dvs}(b^*)y)\oba x\rangle\\&=&\langle x,L^*_{\oba}(R^*_{\dvs}(b^*)y)a^*\rangle=\langle (R_{\oba}(x)\ot \id)\Delta_{\dva}(y),a^*\ot b^*\rangle.\\
\langle R^*_{\oba}(L^*_{\vds}(a^*)x)b^*,y\rangle
&=&\langle b^*,y\oba(L^*_{\vds}(a^*)x)\rangle=\langle L^*_{\oba}(y)b^*,L^*_{\vds}(a^*)x\rangle\\&=&\langle x,a^*\vds(L^*_{\oba}(y)b^*)\rangle=\langle (\id \ot L_{\oba}(y))\Delta_{\vda}(x),a^*\ot b^*\rangle.
		\end{eqnarray*}	
Thus,  \meqref{eq:ppd1}  $\Leftrightarrow$   \meqref{eq:ppd14} $\Leftrightarrow$  \meqref{eq:bi12}:  
\begin{equation}	
\Delta_{\dva}(y\oba x )=\Delta_{\vda}(y\oba x)=(R_{\oba}(x)\ot \id)\Delta_{\dva}(y)+(\id \ot L_{\oba}(y))\Delta_{\vda}(x).
\mlabel{eq:bi12}
\end{equation}	
For  \meqref{eq:ppd2}, we get
		\begin{eqnarray*}	
		\langle R^*_{\obas}(a^*)(x\dva y),b^*\rangle	&=&\langle x\dva y,b^*\obas a^* \rangle=\langle	L^*_{\obas}(b^*)(x\dva y),a^*  \rangle=\langle \Delta_{\oba}(x\dva y),b^*\ot a^*\rangle.\\
			\langle  R^*_{\obas}(a^*)(x\vda y),b^* \rangle	&=&\langle x\vda y,b^*\obas a^* \rangle=\langle L^*_{\obas}(b^*)(x\vda y),a^*   \rangle=\langle \Delta_{\oba}(x\vda y),b^*\ot a^* \rangle.\\
		\langle 	(R^*_{\obas}(a^*)x)\dva y,b^*\rangle&=&\langle R^*_{\obas}(a^*)x, R^*_{\dva}(y)b^*\rangle=\langle x,(R^*_{\dva}(y)b^*)\obas a^*\rangle\\&=&\langle a^*, L^*_{\obas}(R^*_{\dva}(y)b^*)x\rangle=\langle (R_{\dva}(y)\ot \id)\Delta_{\oba}(x),b^*\ot a^*\rangle.\\
		\langle R^*_{\obas}(L^*_{\vda}(x)a^*)y,b^*\rangle&=&	\langle y,b^*\obas(L^*_{\vda}(x)a^*)\rangle=\langle L^*_{\obas}(b^*)y,L^*_{\vda}(x)a^*\rangle\\&=&\langle x\vda(L^*_{\obas}(b^*)y),a^*\rangle=\langle (\id \ot L_{\vda}(x))\Delta_{\oba}(y),b^*\ot a^*\rangle.
				\end{eqnarray*}	
Consequently,  \meqref{eq:ppd2} $\Leftrightarrow$   \meqref{eq:ppd13}$\Leftrightarrow$   \meqref{eq:bi14}:
\begin{equation}	
\Delta_{\oba}(x\dva y)= \Delta_{\oba}(x\vda y) =(R_{\dva}(y)\ot \id) \Delta_{\oba}(x)+(\id \ot L_{\vda}(x))\Delta_{\oba}(y).
\mlabel{eq:bi14}
\end{equation}	
For  \meqref{eq:ppd11}, we have
					\begin{eqnarray*}	
						\langle	(L^*_{\obas}(a^*)x)\dva y,b^* \rangle&=&\langle 	\sigma((\id\ot R_{\dva}(y))\Delta_{\oba}(x)),b^*\ot a^*\rangle.\\
						\langle R^*_{\obas}(R^*_{\dva}(x)a^*)y,b^*\rangle&=&\langle (\id\ot R_{\dva}(x))\Delta_{\oba}(y),b^*\ot a^*\rangle.\\	
\langle L^*_{\obas}(L^*_{\vda}(y)a^*)x,b^* \rangle &=&\langle \sigma((L_{\vda}(y)\ot \id)\Delta_{\oba}(x),b^*\ot a^*\rangle.\\
\langle x\vda(R^*_{\obas}(a^*)y),b^*\rangle&=&\langle (L_{\vda}(x)\ot \id)\Delta_{\oba}(y),b^*\ot a^*\rangle.
							\end{eqnarray*}	
Hence \meqref{eq:ppd11} $\Leftrightarrow$  \meqref{eq:bi15}: 
\begin{equation}	
\sigma((\id\ot R_{\dva}(y))\Delta_{\oba}(x))-(\id\ot R_{\dva}(x))\Delta_{\oba}(y)=\sigma((L_{\vda}(y)\ot \id)\Delta_{\oba}(x))-(L_{\vda}(x)\ot \id)\Delta_{\oba}(y).
\mlabel{eq:bi15}
\end{equation}
For  \meqref{eq:ppd12}, we obtain
				\begin{eqnarray*}	
					\langle(L^*_{\oba}(x)a^*)\dvs b^*, y\rangle&=&\langle \sigma((L_{\oba}(x)\ot \id)\Delta_{\dva} (y)),b^*\ot a^*\rangle.\\
						\langle R^*_{\oba}(R^*_{\dvs}(a^*)x)b^*,y \rangle&=&\langle (L_{\oba}(y)\ot \id)\Delta_{\dva} (x),b^*\ot a^*\rangle.\\
\langle L^*_{\oba}(L^*_{\vds}(b^*)x)a^* ,y\rangle&=&\langle (\id\ot R_{\oba}(y))\Delta_{\vda}(x),b^*\ot a^*\rangle.\\
						\langle a^*\vds (R^*_{\oba}(x)b^*),y \rangle&=&\langle \sigma((\id\ot R_{\oba}(x))\Delta_{\vda}(y)),b^*\ot a^*\rangle.
						\end{eqnarray*}	
This yields \meqref{eq:ppd12}  $\Leftrightarrow$   \meqref{eq:bi16}:  
\begin{equation}	
\sigma((L_{\oba}(x)\ot \id)\Delta_{\dva} (y))-(L_{\oba}(y)\ot \id)\Delta_{\dva} (x)=(\id\ot R_{\oba}(y))\Delta_{\vda}(x)-\sigma((\id\ot R_{\oba}(x))\Delta_{\vda}(y)).
\mlabel{eq:bi16}
\end{equation}	

By arguments analogous to those used above,
one verifies that
\begin{enumerate}
\item  \meqref{eq:ppd3} $\Leftrightarrow$    \meqref{eq:ppd16} $\Leftrightarrow$  \meqref{eq:bi2};
\item \meqref{eq:ppd4} $\Leftrightarrow$   \meqref{eq:ppd15}$\Leftrightarrow$  \meqref{eq:bi8};
\item  \meqref{eq:ppd5} $\Leftrightarrow$  \meqref{eq:ppd18}$\Leftrightarrow$  \meqref{eq:bi3}-\meqref{eq:bi4};
\item   \meqref{eq:ppd6} $\Leftrightarrow$    \meqref{eq:ppd17}$\Leftrightarrow$   \meqref{eq:bi5}-\meqref{eq:bi6};
\item  \meqref{eq:ppd7} $\Leftrightarrow$  \meqref{eq:ppd8}$\Leftrightarrow$   \meqref{eq:bi9};
\item  \meqref{eq:ppd9}  $\Leftrightarrow$   \meqref{eq:ppd10} $\Leftrightarrow$ \meqref{eq:bi10}.
\end{enumerate}
\smallskip

\noindent
($\Rightarrow$) If $(A,A^*,-R^*_{\oba},L^*_{\vda},R^*_{\dva},L^*_{\oba},-R^*_{\obas},L^*_{\vds},R^*_{\dvs},L^*_{\obas})$ is a matched pair of \das, then ~\meqref{eq:bi2}--\meqref{eq:bi10} hold, and hence  $(A,\dva,\vda,\Delta_{\dva},\Delta_{\vda})$ is a \dabi.

\smallskip

\noindent
($\Leftarrow$) Suppose that $(A,\dva,\vda,\Delta_{\dva},\Delta_{\vda})$ is a \dabi. 
Indeed, \eqref{eq:bi14} follows by subtracting \eqref{eq:bi8} from \eqref{eq:bi10} and \eqref{eq:bi9} from \eqref{eq:bi2}, together with the identity
$
\Delta_{\oba}=\Delta_{\vda}-\Delta_{\dva}.
$
Similarly, \eqref{eq:bi12} can be obtained.
Moreover, from ~\meqref{eq:bi3}--\meqref{eq:bi6}, we obtain ~\meqref{eq:bi15} and~\meqref{eq:bi16}.
Thus,   \meqref{eq:bi2}--\meqref{eq:bi10} imply that   \meqref{eq:ppd1}--\meqref{eq:ppd18} are satisfied. So $(A,A^*,-R^*_{\oba},L^*_{\vda},R^*_{\dva},L^*_{\oba},-R^*_{\obas},L^*_{\vds},R^*_{\dvs},L^*_{\obas})$ is a matched pair of \das.
\end{proof}

Combining Theorem~\mref{thm:Manin} and Theorem~\mref{thm:matched}, we obtain
\begin{coro}\mlabel{coro:equiv}
	Let $(A,\dva,\vda)$ and $(A^*,\dvs,\vds)$ be \das. Then the following statements are equivalent:
	\begin{enumerate}
		\item$(A,A^*,-R^*_{\oba},L^*_{\vda},R^*_{\dva},L^*_{\oba},-R^*_{\obas},L^*_{\vds},R^*_{\dvs},L^*_{\obas})$ is a matched pair of \das.\mlabel{r1}
		\item There exists a standard Manin triple of \das associated to $(A,\dva,\vda)$ and $(A^*,\dvs,\vds)$.\mlabel{r2}
		\item $(A,\dva,\vda,\Delta_{\dva},\Delta_{\vda})$ is a \dabi, where $\Delta_{\dva}$ is the linear dual of $\dvs$ and 
		$\Delta_{\vda}$ is the linear dual of $\vds$.
		\mlabel{r3}
	\end{enumerate}
\end{coro}

\begin{exam}
 Let $A$ be a two-dimensional vector space with basis $\{e_1, e_2\}$. Define 
the binary operations $\dva$ and $\vda$ by the following Cayley tables.
$$
		\begin{minipage}[t]{0.3\textwidth}
			\centering
			\begin{tabular}{c|cc}
				$\dva$	& $e_1$ &$ e_2$\\
				\hline 	
				$e_1$ & $e_1$ & $0$   \\
				$e_2$ & $e_2$ & $0$  \\
			\end{tabular}
		\end{minipage}
		\hspace{0.02cm}
		\begin{minipage}[t]{0.3\textwidth}
			\centering
				\begin{tabular}{c|cc}
					$\vda$	& $e_1$ & $e_2$  \\
					\hline 	
					$e_1$ & $e_1$ & $0$ \\
					$e_2$ & $0$ & $0$ \\
			\end{tabular}
		\end{minipage}
$$
Moreover, define  coproducts $\Delta_{\dva}$ and $\Delta_{\vda}$ by
$$ \Delta_{\dva}(e_1)=\Delta_{\vda}(e_1)=e_1\ot e_1, \quad\Delta_{\dva}(e_2)=\Delta_{\vda}(e_2)=e_1\ot e_2.$$
It follows by direct computation that $(A, \Delta_{\dva}, \Delta_{\vda})$ is a \dca, and 
 $(A, \dva, \vda, \Delta_{\dva}, \Delta_{\vda})$ is a \dabi.
\end{exam}

\section{Triangular diassociative bialgebras, diassociative Yang-Baxter equations and $\mathcal{O}$-operators}
\mlabel{sec:4}
In this section, we introduce the notion of the diassociative Yang-Baxter equation (DYBE) in a \da, and then show that every symmetric solution of the DYBE give rise to a  \dabi. We give the notion of an $\mathcal{O}$-operator on a \da. 
We prove that any $\mathcal{O}$-operator on a \da
gives rise to a symmetric solution of the DYBE in the semi-direct product \da. The notion of a pre-\da is introduced to interpret $\mathcal{O}$-operators and construct solutions of the DYBE.

\subsection{Diassociative Yang-Baxter equations and triangular diassociative bialgebras}
\begin{defn}
	Let $(A,\dva,\vda)$ be a \da and $r=\sum\limits_{i}u_{i}\otimes v_{i}\in A\otimes A$. Define
	\begin{equation*}
		r_{12}\vda r_{13}=\sum\limits_{i,j}u_i\vda u_j\otimes v_i \otimes v_j,\;
		r_{13}\dva r_{23}=\sum\limits_{i,j}u_i\otimes u_j \otimes v_i \dva v_j,\;
		r_{23}\oba r_{12}=\sum\limits_{i,j}u_i\otimes u_j\oba v_i\otimes v_j.
	\end{equation*}
	The equation $$D(r):=r_{12}\vda r_{13}-r_{13}\dva r_{23}-r_{23}\oba r_{12}=0$$ is called the {\bf diassociative Yang-Baxter equation} (or the {\bf DYBE} in short) in a \da $(A,\dva,\vda)$. If $D(r)=0$, then we say that $r$ is a {\bf solution} of the \dybe in $(A,\dva,\vda)$.
	By definition, one can also write down an explicit presentation of the DYBE as follows:
	\begin{equation}\mlabel{eq:dybe}
		D(r)=\sum\limits_{i,j}u_i\vda u_j\otimes v_i \otimes v_j-u_i\otimes u_j \otimes v_i \dva v_j-u_i\otimes u_j\oba v_i \otimes v_j=0.
	\end{equation}
\end{defn}
\begin{prop}\mlabel{prop:triangular}
	Let $(A,\dva,\vda)$ be a \da and let $r=\sum\limits_{i} u_{i}\otimes v_{i}\in A\otimes A$ be a symmetric solution of the DYBE in $(A,\dva,\vda)$. Define linear maps $\Delta_{\dva},\Delta_{\vda}$  by 
	\begin{eqnarray}
		&&\Delta_{\dva}(x):=(\id \otimes L_{\oba}(x)-R_{\vda}(x)\otimes\id)r,\mlabel{eq:dx1}\\
		&&\Delta_{\vda}(x):=(\id \otimes L_{\dva}(x) +R_{\oba}(x)\otimes\id)(-r),\quad x\in A.\mlabel{eq:dx2}
	\end{eqnarray}
Then $(A,\dva,\vda,\Delta_{\dva},\Delta_{\vda})$ is a \dabi,  called a {\bf triangular diassociative bialgebra}.
\end{prop}
\begin{proof}
We first show that $(A,\Delta_{\dva},\Delta_{\vda})$ is a \dca, that is, $\Delta_{\dva}$ and $\Delta_{\vda}$ satisfy ~\meqref{eq:co1}-\meqref{eq:co3}. We only verify \eqref{eq:co1}, since the proofs of
\eqref{eq:co2} and \eqref{eq:co3} are analogous.
Let $a\in A$. Then
	\begin{eqnarray*}
		&&(\id\ot\Delta_{\dva} )\Delta_{\dva}(a)\\
&\stackrel{\eqref{eq:dx1}}{=}&(\id\ot\Delta_{\dva} )(\sum_{i}u_i\ot (a \vda v_i-a\dva v_i) -u_i\vda a \ot v_i)\\
		&=&\sum_{i}u_i\ot \Delta_{\dva}(a \vda v_i-a\dva v_i)-u_i\vda a\ot \Delta_{\dva}(v_i)\\
		&=&\sum_{i,j}u_i\ot(u_j\ot ((a\vda v_i-a\dva v_i)\vda v_j-(a\vda v_i-a\dva v_i)\dva v_j)\\
		&&-u_j\vda (a\vda v_i-a\dva v_i) \ot v_j)	-u_i\vda a \ot (u_j\ot (v_i\vda v_j-v_i\dva v_j)-u_j\vda v_i \ot v_j)\\
		&=&\sum_{i,j}u_i\ot u_j\ot (a\vda v_i)\vda v_j-u_i\ot u_j \ot (a\dva v_i)\vda v_j-u_i\ot u_j\ot(a\vda v_i)\dva v_j\\
&&+u_i\ot u_j \ot (a\dva v_i)\dva v_j-u_i\ot u_j\vda(a\vda v_i)\ot v_j+u_i\ot u_j\vda(a\dva v_i)\ot v_j\\
&&-u_i\vda a\ot u_j\ot v_i\vda v_j+u_i\vda a\ot u_j\ot v_i\dva v_j+u_i\vda a\ot u_j\vda v_i\ot v_j\\
&\stackrel{\eqref{eq:dia3}}{=}&\sum_{i,j}-u_i\ot u_j\ot(a\vda v_i)\dva v_j+u_i\ot u_j \ot (a\dva v_i)\dva v_j-u_i\ot u_j\vda(a\vda v_i)\ot v_j\\
&&+u_i\ot u_j\vda(a\dva v_i)\ot v_j+(R_{\vda}(a)\ot \id \ot \id)\sigma(123)(-u_j\ot v_i\vda v_j\ot u_i\\
&&+u_j\ot v_i\dva v_j \ot u_i+u_j\vda v_i\ot v_j \ot u_i)\\
&\stackrel{(\eqref{eq:dybe},r=\sigma(r))}{=}&\sum_{i,j}-u_i\ot u_j\ot(a\vda v_i)\dva v_j+u_i\ot u_j \ot (a\dva v_i)\dva v_j-u_i\ot u_j\vda(a\vda v_i)\ot v_j\\
&&+u_i\ot u_j\vda(a\dva v_i)\ot v_j+(R_{\vda}(a)\ot \id \ot \id)\sigma(123)(u_i\ot u_j\ot v_i\dva v_j)\\
&=&\sum_{i,j}-u_i\ot u_j\ot(a\vda v_i)\dva v_j+u_i\ot u_j \ot (a\dva v_i)\dva v_j-u_i\ot u_j\vda(a\vda v_i)\ot v_j\\
&&+u_i\ot u_j\vda(a\dva v_i)\ot v_j+(v_i\dva v_j)\vda a\ot u_i\ot u_j. 
			\end{eqnarray*}
\vsd
			\begin{eqnarray*}
		&&(\id\ot\Delta_{\vda} )\Delta_{\dva}(a)\\
&\stackrel{\eqref{eq:dx1}}{=}&(\id\ot\Delta_{\vda} )(\sum_{i}u_i\ot (a \vda v_i-a\dva v_i) -u_i\vda a \ot v_i)\\
&=&\sum_{i}u_i\ot \Delta_{\vda}(a \vda v_i-a\dva v_i)-u_i \vda a\ot \Delta_{\vda}(v_i)\\
&\stackrel{\eqref{eq:dx2}}{=}&
		\sum_{i,j}u_i\ot((u_j\dva (a\vda v_i-a\dva v_i)-u_j\vda(a\vda v_i-a\dva v_i))\ot v_j\\
&&-u_j\ot (a\vda v_i-a\dva v_i)\dva v_j)-u_i\vda a\ot((u_j\dva v_i-u_j\vda v_i)\ot v_j-u_j\ot v_i\dva v_j)\\
&=&\sum_{i,j}u_i\ot u_j\dva (a\vda v_i)\ot v_j-u_i\ot u_j\dva(a\dva v_i)\ot v_j-u_i\ot u_j\vda(a\vda v_i)\ot v_j\\
&&+u_i\ot u_j\vda(a \dva v_i)\ot v_j-u_i\ot u_j\ot(a\vda v_i)\dva v_j+u_i\ot u_j\ot(a\dva v_i)\dva v_j\\
&&-u_i\vda a\ot u_j\dva v_i\ot v_j+u_i\vda a\ot u_j\vda v_i\ot v_j+u_i\vda a\ot u_j\ot v_i\dva v_j\\
&=&\sum_{i,j}-u_i\ot u_j\vda(a\vda v_i)\ot v_j+u_i\ot u_j\vda(a \dva v_i)\ot v_j-u_i\ot u_j\ot(a\vda v_i)\dva v_j\\
&&+u_i\ot u_j\ot(a\dva v_i)\dva v_j+(R_{\vda}(a)\ot \id \ot \id)(-u_i\ot u_j\dva v_i\ot v_j+u_i\ot u_j\vda v_i\ot v_j\\
&&+u_i\ot u_j\ot v_i\dva v_j)\\
&=&\sum_{i,j}-u_i\ot u_j\vda(a\vda v_i)\ot v_j+u_i\ot u_j\vda(a \dva v_i)\ot v_j-u_i\ot u_j\ot(a\vda v_i)\dva v_j\\
&&+u_i\ot u_j\ot(a\dva v_i)\dva v_j+(R_{\vda}(a)\ot \id \ot \id)(v_i\vda v_j\ot u_i\ot u_j)\\
&=&\sum_{i,j}-u_i\ot u_j\vda(a\vda v_i)\ot v_j+u_i\ot u_j\vda(a \dva v_i)\ot v_j-u_i\ot u_j\ot(a\vda v_i)\dva v_j\\
&&+u_i\ot u_j\ot(a\dva v_i)\dva v_j+(v_i\vda v_j)\vda a\ot u_i\ot u_j\\
&=&\sum_{i,j}-u_i\ot u_j\vda(a\vda v_i)\ot v_j+u_i\ot u_j\vda(a \dva v_i)\ot v_j-u_i\ot u_j\ot(a\vda v_i)\dva v_j\\
&&+u_i\ot u_j\ot(a\dva v_i)\dva v_j+(v_i\dva v_j)\vda a\ot u_i\ot u_j.
		\end{eqnarray*}
Moreover, we have
		\begin{eqnarray*}
		&&(\Delta_{\dva} \ot \id)\Delta_{\dva}(a)\\&=&	(\Delta_{\dva} \ot \id)(\sum_{i}u_i\ot (a \vda v_i-a\dva v_i) -u_i\vda a \ot v_i)\\
		&=&\sum_{i}\Delta_{\dva}(u_i)\ot (a \vda v_i-a\dva v_i) -\Delta_{\dva}(u_i\vda a) \ot v_i\\
		&=&\sum_{i,j}u_j\ot (u_i\vda v_j-u_i\dva v_j)-u_j\vda u_i\ot v_j)\ot (a \vda v_i-a\dva v_i)\\
		&&-(u_j\ot ((u_i\vda a)\vda v_j-(u_i\vda a)\dva v_j)-u_j\vda (u_i\vda a)\ot v_j)\ot v_i\\
		&=&\sum_{i,j}u_j\ot u_i\vda v_j\ot a\vda v_i-u_j\ot u_i \vda v_j\ot a\dva v_i-u_j\ot u_i\dva v_j\ot a\vda v_i\\
		&&+u_j\ot u_i\dva v_j\ot a\dva v_i-u_j\vda u_i\ot v_j\ot a\vda v_i+u_j\vda u_i\ot v_j \ot a\dva v_i\\
		&&-u_j\ot (u_i\vda a)\vda v_j\ot v_i+u_j\ot (u_i\vda a)\dva v_j\ot v_i+u_j\vda(u_i\vda a)\ot v_j\ot v_i\\
		&=&\sum_{i,j}(\id\ot \id \ot L_{\vda}(a))(u_j\ot u_i\vda v_j\ot v_i-u_j\ot u_i\dva v_j\ot  v_i-u_j\vda u_i\ot v_j\ot v_i)\\
		&&+(\id\ot \id \ot L_{\dva}(a))(-u_j\ot  u_i \vda v_j\ot  v_i+u_j\ot u_i\dva v_j\ot  v_i+u_j\vda u_i\ot v_j \ot  v_i)\\
		&&-u_j\ot (u_i\vda a)\vda v_j\ot v_i+u_j\ot (u_i\vda a)\dva v_j\ot v_i+u_j\vda(u_i\vda a)\ot v_j\ot v_i\\
		&=&\sum_{i,j}(\id\ot \id \ot L_{\vda}(a))(-u_i\ot u_j \ot v_i\dva v_j)+(\id\ot \id \ot L_{\dva}(a))(u_i\ot u_j\ot v_i\dva v_j)\\
		&&-u_j\ot (u_i\vda a)\vda v_j\ot v_i+u_j\ot (u_i\vda a)\dva v_j\ot v_i+u_j\vda(u_i\vda a)\ot v_j\ot v_i\\
		&=&\sum_{i,j}-u_i\ot u_j \ot a\vda (v_i\dva v_j)+u_i\ot u_j \ot a\dva(v_i\dva v_j)-u_j\ot (u_i\vda a)\vda v_j\ot v_i\\
		&&+u_j\ot (u_i\vda a)\dva v_j\ot v_i+u_j\vda(u_i\vda a)\ot v_j\ot v_i\\
		&=&\sum_{i,j}-u_i\ot u_j \ot (a\vda v_i)\dva v_j+u_i\ot u_j \ot (a\dva v_i)\dva v_j-u_i\ot u_j\vda (a\vda v_i)\ot v_j\\
		&&+u_i\ot u_j\vda (a\dva v_i)\ot v_j+(v_i\dva v_j)\vda a\ot u_i\ot u_j.
	\end{eqnarray*}
	Thus,   \meqref{eq:co1} holds. 

To verify ~\eqref{eq:bi2}--\eqref{eq:bi10}, it suffices to prove ~\eqref{eq:bi2}, as the others follow analogously. For all $x,y \in A$, we have
	\begin{align*}
		\Delta_{\vda} (x\dva y) &= \sum_{i}(u_i \dva (x\dva y) - u_i \vda (x\dva y))\ot v_i - u_i\ot (x\dva y)\dva v_i \\
		&= \sum_{i} u_i \dva (x\dva y)\ot v_i - u_i \vda (x\dva y)\ot v_i - u_i\ot (x\dva y) \dva v_i.
	\end{align*}
	\begin{eqnarray*}
		&&(\id\ot L_{\dva} (x))\Delta_{\vda}(y) + ((R_{\dva} - R_{\vda})(y)\ot \id)\Delta_{\vda} (x)\\
		&=& \sum_{i} (\id \ot L_{\dva} (x))((u_i\dva y - u_i\vda y)\ot v_i - u_i \ot y \dva v_i)\\
		&&+ ((R_{\dva} -R_{\vda})(y)\ot \id )((u_i \dva x - u_i\vda x)\ot v_i - u_i \ot x \dva v_i)\\
		&=& \sum_{i}u_i \dva y\ot x\dva v_i  - u_i\vda y \ot x \dva v_i -u_i \ot x \dva (y\dva v_i)+ (u_i \dva x)\dva y\ot v_i\\
		&&- (u_i\vda x)\dva y\ot v_i - u_i \dva y\ot x \dva v_i - (u_i \dva x)\vda y\ot v_i + (u_i \vda x )\vda y\ot v_i\\
&&+ u_i \vda y\ot x \dva v_i\\
		&=& \sum_{i} -u_i\ot x \dva (y\dva v_i) + (u_i \dva x)\dva y\ot v_i - (u_i \vda x)\dva y\ot v_i\\
		&=& \sum_{i} - u_i\ot (x\dva y) \dva v_i+u_i \dva (x\dva y)\ot v_i - u_i \vda (x\dva y)\ot v_i .
		\end{eqnarray*}	
Similarly, we get
		\begin{eqnarray*}
		&&(\id\ot L_{\dva} (x))(\Delta_{\vda} - \Delta_{\dva} )(y)+ (R_{\dva}(y)\ot \id )\Delta_{\vda} (x)\\
		&=& \sum_{i} - u_i\ot (x\dva y) \dva v_i+u_i \dva (x\dva y)\ot v_i - u_i \vda (x\dva y)\ot v_i .
		\end{eqnarray*}
Hence,  \meqref{eq:bi2}  holds.
\end{proof}
\subsection{Diassociative Yang-Baxter equations and $\mathcal{O}$-operators of diassociative algebras}
\begin{defn}\mlabel{defn:ope}
	Let $(A,\dva,\vda)$ be a \da and  $(V,\ell_{\dva}, r_{\dva}, \ell_{\vda}, r_{\vda})$ be a representation of $(A,\dva,\vda)$. A linear map
	$T: V\to A$ is called an {\bf $\mathcal{O}$-operator} of $(A,\dva,\vda)$  associated  to  $(V,\ell_{\dva}, r_{\dva}, \ell_{\vda}, r_{\vda})$, if for all $	u,v\in V$, $T$ satisfies
	\begin{equation}\mlabel{eq:oper}
		T(u)\dva T(v)=T(\ell_{\dva}(T(u))v+r_{\dva}(T(v))u),\,
		T(u)\vda T(v)=T(\ell_{\vda}(T(u))v+r_{\vda}(T(v))u).
	\end{equation}
\end{defn}
For a vector space $V$, the isomorphism $V\otimes V\cong {\rm Hom}(V^*,V)$ identifies an $r\in V\otimes V$ with a linear map $r:V^*\to V$. Explicitly,  if $r=\sum\limits_{i}u_{i}\otimes v_{i}$, then the corresponding map $r$ is given by
\begin{equation}\mlabel{eq:identify}
	r:V^{*}\rightarrow V,\;\; r(u^{*})=\sum_{i}\langle u^{*}, u_{i}\rangle v_{i},\, u^{*}\in V^{*}.
\end{equation}
\begin{prop}
	Let $(A,\dva,\vda)$ be a \da 
	and $r=\sum\limits_{i} u_{i}\otimes v_{i}\in A\otimes A$ be symmetric.
	Then $r$ is a solution of the DYBE in $(A,\dva,\vda)$ if and only if
		${r}$ is an $\mathcal{O}$-operator of $(A,\dva,\vda)$ associated to
		$(A^*, -R^*_{\oba},L^*_{\vda},R^*_{\dva},L^*_{\oba})$.
\end{prop}
\begin{proof}
	By Example~\mref{ex:dual}, $(A^*,-R^*_{\oba},L^*_{\vda},R^*_{\dva},L^*_{\oba})$ is a representation of the \da  $(A,\dva,\vda)$.
	For all $a^*,b^*, c^*\in A^*$, by ~\meqref{eq:identify}, we have
	$$
	\sqmon{r(a^*)\vda r(b^*),c^*}=\sqmon{\sum_i\sqmon{a^*,u_i}v_i\dva\sum_j \sqmon{b^*,u_j}v_j,c^*}
	=\sqmon{\sum_{i,j}u_i\ot u_j\ot v_i\dva v_j, a^*\ot b^*\ot c^*}.$$
Moreover, by ~\meqref{eq:rho-star}, we obtain
	\begin{align*}
		\sqmon{r(-{R}^{*}_{\oba}(r(a^*))b^*),c^*}&=\sqmon{\sum_i\sqmon{-{R}^{*}_{\oba}(r(a^*))b^*,u_i}v_i,c^*}
		=\sqmon{\sum_i\sqmon{b^*,-{R}_{\oba}(r(a^*))u_i}v_i,c^*}\\
		&=\sum_i\sqmon{b^*,-u_i\oba r(a^*)}\sqmon{v_i,c^*}
		=\sum_i\sqmon{b^*, -u_i\oba \sum_{j}\sqmon{a^*,u_j}v_j}\sqmon{v_i,c^*}\\
		&=\sum_{i,j}\sqmon{a^*,u_j}\sqmon{b^*, -u_i\oba v_j}\sqmon{v_i,c^*}
		=\sqmon{\sum_{i,j} u_j\ot (-u_i\oba v_j)\ot v_i, a^*\ot b^*\ot c^*}\\
	&=-\sqmon{\sum_{i,j} u_i\ot u_j\oba v_i\ot v_j, a^*\ot b^*\ot c^*}.\\
		\sqmon{r({L}^{*}_{\vda}(r(b^*))a^*),c^*}&=\sqmon{\sum_i\sqmon{L^*_{\vda}(r(b^*))a^*,u_i}v_i,c^*}
		=\sum_i\sqmon{a^*,r(b^*)\vda u_i}\sqmon{v_i,c^*}\\
		&=\sum_i\sqmon{a^*, \sum_{j}\sqmon{b^*,u_j}v_j\vda u_i}\sqmon{v_i,c^*}
		=\sqmon{\sum_{i,j} v_j\vda u_i\ot u_j\ot v_i, a^*\ot b^*\ot c^*}\\
		&\stackrel{(r=\sigma(r), i\leftrightarrow j)}{=}\sqmon{\sum_{i,j} u_i\vda u_j\ot v_i\ot v_j, a^*\ot b^*\ot c^*}.
	\end{align*}
Then by ~\meqref{eq:dybe}, $r(a^*)\dva r(b^*)=r((-{R}^{*}_{\oba})(r(a^*))b^*+{L}^{*}_{\vda}(r(b^*))a^*)$ if and only if $D(r)=0$.
Similarly, we can show that $r(a^*)\vda r(b^*)=r(R^{*}_{\dva}(r(a^*))b^*+L^{*}_{\oba}(r(b^*))a^*)$ if and only if $D(r)=0$. Thus, the statement holds.
\end{proof}
We next prove that $\mathcal{O}$-operators provide numerous solutions of the DYBE in semi-direct
product \das, which, via Proposition~\mref{prop:triangular}, induce \dabis. 
\begin{theorem}\mlabel{thm:semi-dybe}
 Let $(A,\dva,\vda)$ be a \da.	Let $(V,\ell_{\dva},r_{\dva},\ell_{\vda},r_{\vda})$ be a representation of  $(A,\dva,\vda)$. Let $T:V\rightarrow A$ be a linear map identified as an element in 
	$ (A\ltimes_{r^*_{\dva},\ell^*_{\vda}-\ell^*_{\dva}}^{r^*_{\dva}-r^*_{\vda},\ell^*_{\vda}}V^*)\otimes
	(A\ltimes_{r^*_{\dva},\ell^*_{\vda}-\ell^*_{\dva}}^{r^*_{\dva}-r^*_{\vda},\ell^*_{\vda}}V^*).$  
	Then $r:=T+\sigma(T)$ is a symmetric solution of  the DYBE in the semi-direct product  \da $A\ltimes_{r^*_{\dva},\ell^*_{\vda}-\ell^*_{\dva}}^{r^*_{\dva}-r^*_{\vda},\ell^*_{\vda}}V^*$ defined by Definition~\mref{def:semi}, if and only if $T$ is an $\mathcal{O}$-operator of  $(A,\dva,\vda)$ associated to $(V,\ell_{\dva},r_{\dva},\ell_{\vda},r_{\vda})$.
\end{theorem}
\begin{proof}
Let $\{e_1,\ldots,e_n\}$ be a basis of $V$ and let
$\{e_1^*,\ldots,e_n^*\}$ be the corresponding dual basis of $V^*$.	 Then we write  $T=\sum_{i=1}^{n}e_i^*\otimes T(e_i)$. Thus, $r=T+\sigma(T)=\sum_{i=1}^{n}(e_i^* \otimes T(e_i)+T(e_i)\otimes e_i^*)$. Denote the two products on the semi-direct product algebra by
$\dashv'$ and $\vdash'$. So we obtain
	\begin{align*}
		r_{12}\vd' r_{13}
		&=\sum_{i,j}e_i^*\vd'e_j^*\ot T(e_i)\ot T(e_j)+T(e_i)\vd' e^*_j\ot e_i^*\ot T(e_j)\\
		&\quad+e^*_i\vd' T(e_j)\ot T(e_i)\ot e^*_j+T(e_i)\vd' T(e_j)\ot e^*_i\ot e^*_j\\
		&=\sum_{i,j}r_{\dva}^*(T(e_i)) e^*_j\ot e_i^*\ot T(e_j)
		+(\ell^*_{\vda}-\ell^*_{\dva})(T(e_j))e^*_i\ot T(e_i)\ot e^*_j\\
&\quad+T(e_i)\vda T(e_j)\ot e^*_i\ot e^*_j .
\end{align*}
Similarly, we  get
\begin{align*}
		r_{13}\dv' r_{23}
&=\sum_{i,j}e_i^*\ot e_j^*\ot T(e_i)\dva T(e_j)+T(e_i)\ot e_j^* \ot \ell^*_{\vda}(T(e_j))e^*_i\\
&\quad+e^*_i\ot T(e_j)\ot(r^*_{\dva}-r^*_{\vda} ) (T(e_i)) e^*_j.\\
		r_{23}\vd' r_{12}
		&=\sum_{i,j}e_j^*\ot (\ell^*_{\vda}-\ell^*_{\dva})(T(e_j))e_i^* \ot T(e_i)+e^*_j\ot T(e_i)\vda T(e_j)\ot e^*_i\\
&\quad+T(e_j)\ot r^*_{\dva}(T(e_i)) e^*_j\ot e^*_i.\\
		r_{23}\dv' r_{12}
		&=\sum_{i,j}e_j^*\ot \ell^*_{\vda}(T(e_j))e_i^* \ot T(e_i)+e^*_j\ot T(e_i)\dva T(e_j)\ot e^*_i\\
&\quad+T(e_j)\ot (r^*_{\dva}-r^*_{\vda})(T(e_i)) e^*_j\ot e^*_i.
	\end{align*}
	Then $r$ is a symmetric solution of  the DYBE in   $A\ltimes_{\ell^*_{\vda},\ell^*_{\vda}-\ell^*_{\dva}}^{r^*_{\dva}-r^*_{\vda},r^*_{\dva}}V^*$ , if and only if
	\begin{align*}
		&\sum_{i,j}r_{\dva}^*(T(e_i)) e^*_j\ot e_i^*\ot T(e_j)
		+(\ell^*_{\vda}-\ell^*_{\dva})(T(e_j))e^*_i\ot T(e_i)\ot e^*_j+T(e_i)\vda T(e_j)\ot e^*_i\ot e^*_j \\
		&\quad-e_i^*\ot e_j^*\ot T(e_i)\dva T(e_j)-T(e_i)\ot e_j^* \ot \ell^*_{\vda}(T(e_j))e^*_i
		-e^*_i\ot T(e_j)\ot(r^*_{\dva}-r^*_{\vda} ) (T(e_i)) e^*_j\\
		&\quad-e_j^*\ot (\ell^*_{\vda}-\ell^*_{\dva})(T(e_j))e_i^* \ot T(e_i)-e^*_j\ot T(e_i)\vda T(e_j)\ot e^*_i-T(e_j)\ot r^*_{\dva}(T(e_i)) e^*_j\ot e^*_i\\
		&\quad +e_j^*\ot \ell^*_{\vda}(T(e_j))e_i^* \ot T(e_i)+e^*_j\ot T(e_i)\dva T(e_j)\ot e^*_i+T(e_j)\ot (r^*_{\dva}-r^*_{\vda})(T(e_i)) e^*_j\ot e^*_i=0,
	\end{align*}
	which is equivalent to
	\begin{align}
		&\sum_{i,j}	r_{\dva}^*(T(e_i)) e^*_j\ot e_i^*\ot T(e_j)-e_i^*\ot e_j^*\ot T(e_i)\dva T(e_j)\nonumber\\
&\quad -e_j^*\ot (\ell^*_{\vda}-\ell^*_{\dva})(T(e_j))e_i^* \ot T(e_i)+e_j^*\ot \ell^*_{\vda}(T(e_j))e_i^* \ot T(e_i)=0, \mlabel{eq:t1}\\
		&\sum_{i,j}	(\ell^*_{\vda}-\ell^*_{\dva})(T(e_j))e^*_i\ot T(e_i)\ot e^*_j-e^*_i\ot T(e_j)\ot(r^*_{\dva}-r^*_{\vda} ) (T(e_i)) e^*_j\nonumber\\
&\quad -e^*_j\ot T(e_i)\vda T(e_j)\ot e^*_i+e^*_j\ot T(e_i)\dva T(e_j)\ot e^*_i=0,\mlabel{eq:t2}\\
		&\sum_{i,j}	T(e_i)\vda T(e_j)\ot e^*_i\ot e^*_j -T(e_i)\ot e_j^* \ot \ell^*_{\vda}(T(e_j))e^*_i\nonumber\\
&\quad -T(e_j)\ot r^*_{\dva}(T(e_i)) e^*_j\ot e_i^*+T(e_j)\ot (r^*_{\dva}-r^*_{\vda})(T(e_i)) e^*_j\ot e^*_i=0.\mlabel{eq:t3}
	\end{align}
	Note that  \begin{align*}
		r_{\dva}^*(T(e_i)) e^*_j&=\sum_k\langle r_{\dva}^*(T(e_i)) e^*_j,e_k\rangle e_k^*,\qquad (\ell^*_{\vda}-\ell^*_{\dva})(T(e_j))e_i^*=\sum_k\langle  (\ell^*_{\vda}-\ell^*_{\dva})(T(e_j)e_i^*,e_k\rangle e_k^*,\\
		\ell^*_{\vda}(T(e_j))e_i^*&=\sum_k\langle \ell^*_{\vda}(T(e_j))e_i^*,e_k \rangle e_k^*.
	\end{align*}
From the above,  we obtain
	\begin{align*}
		\sum_{i,j}r_{\dva}^*(T(e_i)) e^*_j\ot e_i^*\ot T(e_j)&=\sum_{i,j,k}\langle r_{\dva}^*(T(e_i)) e^*_j,e_k\rangle e_k^*\ot e_i^*\ot T(e_j)\\
		&=\sum_{i,k}e_k^*\ot e_i^*\ot T(\sum_{j}\langle  e^*_j,r_{\dva}(T(e_i))e_k\rangle e_j)\\
&=\sum_{i,k}e_k^*\ot e_i^*\ot T(r_{\dva}(T(e_i))e_k)\\
&\stackrel{(k\leftrightarrow i,  i\leftrightarrow j)}{=}\sum_{i,j}e_i^*\ot e_j^*\ot T(r_{\dva}(T(e_j))e_i).
\end{align*}
Similarly, 
\begin{align*}
		\sum_{i,j}-e_j^*\ot (\ell^*_{\vda}-\ell^*_{\dva})(T(e_j))e_i^* \ot T(e_i)
&=\sum_{i,j}-e_i^*\ot e^*_j\ot T((\ell_{\vda}-\ell_{\dva})(T(e_i))e_j),\\
		\sum_{i,j}e_j^*\ot \ell^*_{\vda}(T(e_j))e_i^* \ot T(e_i)
&=\sum_{i,j}e_i^*\ot  e_j^*\ot T(\ell_{\vda}(T(e_i))e_j ).
	\end{align*}
	Hence  \meqref{eq:t1} holds if and only if 
	$$T(e_i)\dva  T(e_j)=T(\ell_{\dva}(T(e_i)) e_j+r_{\dva}(T(e_j))e_i).$$
Similarly, one can show that
\meqref{eq:t2} is equivalent to
	$$T(e_i)\vda  T(e_j)-T(\ell_{\vda}(T(e_i)) e_j+r_{\vda}(T(e_j))e_i)=T(e_i)\dva  T(e_j)-T(\ell_{\dva}(T(e_i)) e_j+r_{\dva}(T(e_j))e_i).$$
Moreover,  \meqref{eq:t3} holds if and only if 
	$$T(e_i)\vda  T(e_j)=T(\ell_{\vda}(T(e_i)) e_j+r_{\vda}(T(e_j))e_i).\qedhere$$
\end{proof}
\subsection{Pre-diassociative algebras and $\mathcal{O}$-operators}
\begin{defn}
A {\bf pre-diassociative algebra} is a quintuple $(A,\lhd_1, \lhd_2,\rhd_1, \rhd_2)$, where $A$ is a vector space and $\lhd_1, \lhd_2,\rhd_1, \rhd_2: A\ot A\to A$ are binary operations satisfying the following identities for all $x, y, z\in A$:  
\begin{align}
x\lhd_1 (y\lhd_1 z) &= x \lhd_1(y\rhd_1 z)=(x\dva y)\lhd_1 z,\mlabel{eq:pnva1}\\
x\lhd_1  (y\lhd_2 z) &= x \lhd_1(y\rhd_2 z)=(x\lhd_1 y)\lhd_2 z,\mlabel{eq:pnva2}\\
x\lhd_2 (y\dva z) &= x \lhd_2(y\vda z)=(x\lhd_2 y)\lhd_2 z,\mlabel{eq:pnva3}\\
(x\vda y)\lhd_1 z &=x \rhd_1(y\lhd_1 z),\mlabel{eq:pnva4}\\
(x\rhd_1 y)\lhd_2 z &=x \rhd_1(y\lhd_2 z),\mlabel{eq:pnva5}\\
(x\rhd_2 y)\lhd_2 z&=x \rhd_2(y\dva z),\mlabel{eq:pnva6}\\
(x\dva y)\rhd_1 z &=(x\vda y) \rhd_1 z= x\rhd_1 (y\rhd_1 z),\mlabel{eq:pnva7}\\
(x\lhd_1 y)\rhd_2 z &=(x \rhd_1 y)\rhd_2 z=x\rhd_1(y\rhd_2 z),\mlabel{eq:pnva8}\\
(x\lhd_2 y)\rhd_2 z&=(x \rhd_2 y)\rhd_2 z=x\rhd_2 (y\vda z),\mlabel{eq:pnva9}
\end{align}
where
\begin{align}
x\dva y&:=x\lhd_1 y+ x\lhd_2 y, \mlabel{eq:pnvab1}\\
x\vda y&:=x\rhd_1 y+ x\rhd_2 y.\mlabel{eq:pnvab2}
\end{align}
\end{defn}
\begin{exam}Let $(A,\prec,\succ)$ be a dendriform algebra.
If we set $\lhd_1:=\succ, \lhd_2:=\prec$, and $\rhd_1=\rhd_2=0$, then $(A,\lhd_1,\lhd_2,\rhd_1,\rhd_2)$ is a pre-diassociative algebra. Similarly, if we let $\lhd_1=\lhd_2=0$,  $\rhd_1:=\succ$ and $\rhd_2:=\prec$,  then $(A,\lhd_1,\lhd_2,\rhd_1,\rhd_2)$  is also a pre-diassociative algebra. Conversely, a pre-diassociative algebra $(A,\lhd_1,\lhd_2,\rhd_1,\rhd_2)$ with $\lhd_1=\lhd_2=0$ reduces to a dendriform algebra, and similarly, one with $\rhd_1 = \rhd_2 = 0$ also reduces to a dendriform algebra.
\end{exam}

\begin{prop}\mlabel{prop:com-pre-dia}
\begin{enumerate}
\item\mlabel{it:pre-dia}
	Let $(A,\lhd_1, \lhd_2,\rhd_1, \rhd_2)$ be a \pda. Let $\dva, \vda$ be the  binary operations given by ~\meqref{eq:pnvab1} and ~\meqref{eq:pnvab2}. Then the triple $(A, \dva, \vda)$ is a \da , which is called the {\bf sub-adjacent diassociative algebra} of $(A,\lhd_1, \lhd_2,\rhd_1, \rhd_2)$, and $(A,\lhd_1, \lhd_2,\rhd_1, \rhd_2)$  is called a {\bf compatible pre-diassociative algebra} of $(A, \dva, \vda)$.  Furthermore,
    $(A,L_{\lhd_1},R_{\lhd_2},L_{\rhd_1}, R_{\rhd_2})$ is a representation of $(A,\dva,\vda)$. 
\item \mlabel{it:dia-pre}
Conversely, let $A$ be a vector space equipped with binary operations $\lhd_1, \lhd_2,\rhd_1, \rhd_2$. If binary operations $\dva,\vda$ given by ~\meqref{eq:pnvab1} and~\meqref{eq:pnvab2} yield a \da $(A,\dva,\vda)$, and $(A,L_{\lhd_1},R_{\lhd_2},L_{\rhd_1}, R_{\rhd_2})$ is a representation of $(A,\dva,\vda)$, then $(A,\lhd_1,\lhd_2,\rhd_1,\rhd_2)$ is a \pda.
\end{enumerate}
\end{prop}
\begin{proof}
(\mref{it:pre-dia})
If $(A,\lhd_1, \lhd_2,\rhd_1, \rhd_2)$ is a \pda, then ~\meqref{eq:pnva1}--\meqref{eq:pnva9} hold.  The summation of ~\meqref{eq:pnva1}--\meqref{eq:pnva3} gives ~\meqref{eq:dia1}. Similarly, ~\meqref{eq:dia2} follows from ~\meqref{eq:pnva4}--\meqref{eq:pnva6}, and ~\meqref{eq:dia3} follows from ~\meqref{eq:pnva7}--\meqref{eq:pnva9}. Moreover,  $(A,L_{\lhd_1},R_{\lhd_2},L_{\rhd_1},R_{\rhd_2})$ satisfies ~\meqref{eq:repr1}--\meqref{eq:repr9}. Then by Proposition~\mref{pro:repr},  $(A,L_{\lhd_1},R_{\lhd_2},L_{\rhd_1}$, 
$R_{\rhd_2})$ is a representation of $(A,\dva,\vda)$. 
\smallskip

\noindent
(\mref{it:dia-pre})
Suppose that $(A,\dva,\vda)$  is a \da and $(A,L_{\lhd_1},R_{\lhd_2},L_{\rhd_1}, R_{\rhd_2})$ is its representation.  Then ~\meqref{eq:pnva1}--\meqref{eq:pnva9} can be derived from ~\meqref{eq:repr1}--\meqref{eq:repr9} by setting $\ell_{\dva}=L_{\lhd_1},r_{\dva}=R_{\lhd_2},\ell_{\vda}=L_{\rhd_1}, r_{\vda}=R_{\rhd_2}$.
\end{proof}

The proposition below indicates that  $\calo$-operators naturally induce pre-diassociative algebras.
\begin{prop} \mlabel{pro:pnva}
Let $(A,\dva,\vda)$ be a \da, $(V,\ell_{\dva},r_{\dva},\ell_{\vda},r_{\vda})$ be a representation of $(A,\dva,\vda)$ and $T:V\rightarrow A$ be an $\calo$-operator of $(A,\dva,\vda)$ associated to $(V,\ell_{\dva},r_{\dva}, \ell_{\vda},r_{\vda})$. 
Define
\begin{align}
u \lhd_1 v&=\ell_{\dva}(T(u))v,\;
u \lhd_2 v=r_{\dva}(T(v))u,\mlabel{eq:op1}\\
u \rhd_1 v&=\ell_{\vda}(T(u))v,\;
u \rhd_2 v=r_{\vda}(T(v))u,\quad u,v\in V.\mlabel{eq:op2}
\end{align}
If $\dva$ and $\vda$ satisfy ~\meqref{eq:pnvab1} and~\meqref{eq:pnvab2}, then $(V,\lhd_1,\lhd_2,\rhd_1,\rhd_2)$ is a \pda.
\end{prop}

\begin{proof}
	For all $u,v,w \in V$, we have
\begin{align*}
(u\dva v)\lhd_1 w &= \ell_{\dva}(T(u \lhd_1 v+u \lhd_2 v))w
= \ell_{\dva}(T(\ell_{\dva}(T(u)) v+r_{\dva}(T(v))u))w\\
&\stackrel{\meqref{eq:oper}}{=}\ell_{\dva}(T(u)\dva T(v))w\\
&\stackrel{\meqref{eq:repr1}}{=}\ell_{\dva}(T(u))\ell_{\dva}(T(v))w
\stackrel{\meqref{eq:repr1}}{=}\ell_{\dva}(T(u))\ell_{\vda}(T(v))w\\
&=u \lhd_1 (\ell_{\dva}(T(v))w)=u \lhd_1 (\ell_{\vda}(T(v))w)\\
&=u \lhd_1(v\lhd_1 w)=u \lhd_1(v\rhd_1 w),
\end{align*}
proving ~\meqref{eq:pnva1}. Similarly, we can prove ~\meqref{eq:pnva2}--\meqref{eq:pnva9}.
\end{proof}

Conversely, we have
\begin{prop}\mlabel{pro:id} Let $(A,\lhd_1, \lhd_2,\rhd_1, \rhd_2)$ be a \pda  and $(A,\dva,\vda)$ be the sub-adjacent diassociative algebra.
 Then the identity map $\id$ on $A$ is an $\mathcal{O}$-operator of $(A,\dva,\vda)$ associated to the induced representation
$(A,L_{\lhd_1},R_{\lhd_2},L_{\rhd_1}, R_{\rhd_2})$.
\end{prop}
\begin{proof} By ~\meqref{eq:pnvab1}--\meqref{eq:pnvab2}, for all $x,y\in A$,  we have 
$$x\dva y=x\lhd_1 y+x\lhd_2 y=L_{\lhd_1}(x)y+R_{\lhd_2}(y)x,\;
x\vda y=x\rhd_1 y+x\rhd_2 y=L_{\rhd_1}(x)y+R_{\rhd_2}(y)x.$$
By Definition~\mref{defn:ope}, $\id$ is an $\mathcal{O}$-operator of $(A,\dva,\vda)$ associated to $(A,L_{\lhd_1},R_{\lhd_2},L_{\rhd_1}, R_{\rhd_2})$.
\end{proof}

The following is an equivalence of compatible pre-diassociative algebra structures on a diassociative algebra and invertible $\mathcal{O}$-operators.
\begin{prop}
Let $(A, \dva,\vda)$ be a diassociative algebra. Then there exists a compatible pre-diassociative algebra structure $(A,\lhd'_1, \lhd'_2,\rhd'_1, \rhd'_2)$ on $A$  such that ~\meqref{eq:pnvab1} and ~\meqref{eq:pnvab2} hold,
if and only if there exists an invertible $\mathcal{O}$-operator $T: V \to A$ with respect to a representation $(V,\ell_{\dva},r_{\dva},\ell_{\vda}$,
$r_{\vda})$.
Moreover, if such an operator $T$ exists, the compatible pre-diassociative algebra operations on $A$ are explicitly given by:
\begin{align*}
&x \lhd'_1 y := T\left( \ell_{\dva}(x)T^{-1}(y) \right), \, x \lhd'_2 y := T\left( r_{\dva}(y) T^{-1}(x) \right),\\
&x \rhd'_1 y := T\left( \ell_{\vda}(x)T^{-1}(y) \right), \, x \rhd'_2 y := T\left( r_{\vda}(y) T^{-1}(x) \right), \quad \forall x, y \in A. 
\end{align*}
\end{prop}
\begin{proof} Suppose that  $T: V \to A$ is an  invertible $\mathcal{O}$-operator with respect to a representation $(V,\ell_{\dva},r_{\dva},\ell_{\vda},r_{\vda})$. For any $u,v\in V$, we let $x:=T(u)$ and $y:=T(v)$ for some $x,y\in A$. Then by ~\meqref{eq:op1} and ~\meqref{eq:op2}, we obtain
\begin{align*}
u \lhd_1 v&=\ell_{\dva}(T(u))v=T^{-1}(x\lhd_1' y),\;
u \lhd_2 v=r_{\dva}(T(v))u=T^{-1}(x\lhd_2' y),\\
u \rhd_1 v&=\ell_{\vda}(T(u))v=T^{-1}(x\rhd_1' y),\;
u \rhd_2 v=r_{\vda}(T(v))u=T^{-1}(x\rhd_2' y),\quad u,v\in V.
\end{align*}
This gives
\begin{align*}
x\lhd_1' y=T(u)\lhd'_1 T(v)&=T(u\lhd _1 v),\; x\lhd_2' y=T(u)\lhd'_2 T(v)=T(u\lhd _2 v),\\
 x\rhd_1' y =T(u)\rhd'_1 T(v)&=T(u\rhd _1 v), \; x\rhd_2' y=T(u)\rhd'_2 T(v)=T(u\rhd _2 v).
\end{align*}
According to the invertibility of $T$,  $\lhd'_1,\lhd'_2,\rhd'_1,\rhd'_2$ satisfy ~\meqref{eq:pnva1}-\meqref{eq:pnva9}.
By ~\meqref{eq:oper}, we have 
\begin{align*}x\dva y&=T(\ell_{\dva}(x)T^{-1}(y)+r_{\dva}(y)T^{-1}(x))=x\lhd'_1 y+ x\lhd'_2 y,\\
x\vda y&=T(\ell_{\vda}(x)T^{-1}(y)+r_{\vda}(y)T^{-1}(x))=x\rhd'_1 y+ x\rhd'_2 y.
\end{align*}
Then by Proposition~\mref{pro:pnva}, $(A,\lhd'_1, \lhd'_2,\rhd'_1, \rhd'_2)$ is a compatible pre-diassociative algebra of $(A,\dva,\vda)$.

Conversely, if $(A,\lhd'_1, \lhd'_2,\rhd'_1, \rhd'_2)$ is a compatible pre-diassociative algebra of $(A,\dva, \vda)$, by Proposition~\mref{prop:com-pre-dia},  $(A,L_{\lhd'_1},R_{\lhd'_2},L_{\rhd'_1}, R_{\rhd'_2})$ is a representation of $(A,\dva,\vda)$. Moreover, by Proposition~\mref{pro:id},  $\id$ is an $\mathcal{O}$-operator of $(A,\dva,\vda)$ associated to $(A,L_{\lhd'_1},R_{\lhd'_2},L_{\rhd'_1}, R_{\rhd'_2})$.
\end{proof}

For a given pre-diassociative algebra $(A, \lhd_1, \lhd_2,\rhd_1,\rhd_2)$, by Proposition~\mref{prop:com-pre-dia}, $(A,\dva,\vda)$ is a  \da, and   $(A,L_{\lhd_1},R_{\lhd_2},L_{\rhd_1}, R_{\rhd_2})$ is a representation of $(A,\dva,\vda)$.  Then by Proposition~\mref{prop:dual-rep},
$(A^*, R^*_{\lhd_2}-R^*_{\rhd_2},L_{\rhd_1}^*, R_{\lhd_2}^*, L_{\rhd_1}^*-L_{\lhd_1}^*)$ is a representation of $(A,\dva,\vda)$. So by Definition~\mref{def:semi},   $A \ltimes^{R^*_{\lhd_2}-R^*_{\rhd_2},L_{\rhd_1}^*}_{R_{\lhd_2}^*, L_{\rhd_1}^*-L_{\lhd_1}^*} A^*$ is  the semi-direct product diassociative algebra.
Then the following result constructs a nondegenerate symmetric  solution of the DYBE from any pre-diassociative algebra, along with an explicit bilinear form encoding this structure.
\begin{theorem}
\mlabel{thm:pre-dia}
Let $(A, \lhd_1, \lhd_2,\rhd_1,\rhd_2)$ be a pre-diassociative algebra. Let $\{e_1, \dots, e_n\}$ be a basis of $A$, and $\{e_1^*, \dots, e_n^*\}$ its dual basis. Define 
$r:= \sum_{i=1}^n \left( e_i^* \otimes e_i + e_i \otimes e_i^* \right)$. Then
$r$ is a symmetric solution of the DYBE in $A \ltimes^{R^*_{\lhd_2}-R^*_{\rhd_2},L_{\rhd_1}^*}_{R_{\lhd_2}^*, L_{\rhd_1}^*-L_{\lhd_1}^*} A^*$. Furthermore, 
$ r$ is nondegenerate, and 
the bilinear form $\omega_r$ on $A \ltimes^{R^*_{\lhd_2}-R^*_{\rhd_2},L_{\rhd_1}^*}_{R_{\lhd_2}^*, L_{\rhd_1}^*-L_{\lhd_1}^*} A^*$ induced by $r$ is given by:
\begin{equation}\mlabel{eq:inv-omeg}
   \omega_r(x + a^*, y + b^*) = \langle a^*, y \rangle + \langle b^*, x \rangle,\quad x, y \in A,  \;a^*, b^* \in A^*.
\end{equation}
\end{theorem}
\begin{proof}By Proposition~\mref{pro:id}, $\id$ is an $\mathcal{O}$-operator of $(A,\dva,\vda)$ associated to $(A,L_{\lhd_1}, R_{\lhd_2}, L_{\rhd_1}, R_{\rhd_2})$. By Theorem~\mref{thm:semi-dybe}, $r$ is a symmetric solution of the DYBE in  $A \ltimes^{R^*_{\lhd_2}-R^*_{\rhd_2},L_{\rhd_1}^*}_{R_{\lhd_2}^*, L_{\rhd_1}^*-L_{\lhd_1}^*} A^*$. Since $\id$ is invertible, $r$ is nondegenerate. Then, the image of $r$ under the inverse of the linear isomorphism $(A\oplus A^*)\ot (A\oplus A^*)\cong \Hom((A\oplus A^*)^*, A\oplus A^*)$ yields a bilinear form $\omega_r$ via ~\meqref{eq:inv-omeg}, and hence  $\omega_r$ is nondegenerate.
\end{proof}

\section{Applications}
\mlabel{sec:5}
In this section, we show that a Manin triple of diassociative algebras gives rise to a Manin triple of Leibniz algebras. We also show that the tensor product of a quadratic dendriform algebra and a Manin triple of diassociative algebras gives rise to a Manin triple of Lie algebras, and the tensor product of a quadratic perm algebra and a double construction of a Frobenius algebra gives rise to a Manin triple of diassociative algebras.
These correspondences are also interpreted in terms of bialgebras.

\subsection{Construction of Leibniz bialgebras from diassociative bialgebras}

\begin{defn}\cite{BGLZ252} 
 A {\bf Leibniz coalgebra}  is a pair  $(A, \Delta)$, where $A$ is a vector space and  $\Delta: A\to A \otimes A $ is a co-multiplication satisfying the \emph{co-Leibniz identity}
\begin{equation}\mlabel{eq:coleft-Lei}
(\id\ot \Delta)\Delta(x)=(\Delta\ot \id)\Delta(x)+(\sigma\ot\id)(\id\ot \Delta)\Delta(x),\quad x\in A.
\end{equation}
\end{defn}
\begin{prop} 	A diassociative  coalgebra $(A, \Delta_{\dva}, \Delta_{\vda})$ yields a Leibniz coalgebra $(A, \Delta)$, where the coproduct $\Delta$ is defined by 
\begin{equation}\mlabel{eq:dia-coproduct}
\Delta(x):=\Delta_{\vda}(x)-\sigma\Delta_{\dva}(x),\quad x\in A.
\end{equation}
\end{prop}
\begin{proof}This follows from Example~\mref{exam:dia-Lei} and Proposition~\mref{prop:dia-coalg}.
\end{proof}

\begin{defn}\cite{Cha05,TS22} 
A {\bf quadratic Leibniz algebra} is a Leibniz algebra
$(A,\circ_{A})$ equipped with a nondegenerate antisymmetric bilinear form $\omega$ that is 
{\bf invariant}  in the following sense:
 \begin{equation}\mlabel{eq:ql}
   \omega(x,y\circ_{A} z)=\omega(x\circ_{A} z+z\circ_{A} x,y),\quad x,y,z\in A.  
 \end{equation}
\end{defn}
\begin{prop}\mlabel{prop:dia-ind-Lei}
Every quadratic diassociative algebra $(A,\dva,\vda)$ induces a  quadratic Leibniz algebra $(A,\circ_{A})$, where the operation $\circ_{A}$ is given by ~\eqref{eq:dil}. 
\end{prop}
\begin{proof}
By Example~\mref{exam:dia-Lei}, $(A,\circ_{A})$ is a Leibniz algebra. It suffices to verify ~\meqref{eq:ql}. For all $x,y,z\in A$, we have 
\begin{align*}
 \omega(x,y\circ_{A} z)&\stackrel{\eqref{eq:dil}}{=}\omega(x,y\vda z-z\dva y) =\omega(x,y\vda z)-\omega(x,z\dva y) =-\omega(y\vda z,x)-\omega(x,z\dva y)\\
&\stackrel{\eqref{eq:omega}}{=}-\omega(y,z\oba x)-\omega(y,x\oba z)=\omega(z\oba x+x\oba z,y).
\end{align*}
 On the other hand, 
 \begin{align*}
      \omega(x\circ_{A} z+z\circ_{A} x,y)\stackrel{\eqref{eq:dil}}{=}\omega(x\vda z-z\dva x+ z\vda x-x\dva z,y) =\omega(x\oba z+ z\oba x,y).
 \end{align*} 
 Thus, ~\eqref{eq:ql} holds.
\end{proof}
\begin{defn}\mlabel{defn:Leibniz}\mcite{TS22}
		A {\bf Manin triple of Leibniz algebras} is a collection $((A\oplus B,\circ_{d},\omega),(A,\circ_{A}),(B$,
		$\circ_{B}))$, such that $(A\oplus B,\circ_{d},\omega)$ is a quadratic Leibniz algebra, $(A,\circ_{A}),(B,\circ_{B})$ are Leibniz subalgebras of $(A\oplus B,\circ_{d})$ and $A,B$ are isotropic with respect to $\omega$.
		A Manin triple of Leibniz algebras is called {\bf standard} if $B=A^{*}$ and $\omega=\omega_{d}$ is given by \meqref{eq:omd}.
\end{defn}
\begin{prop}\mlabel{prop:dia-manin-lei}
A Manin triple $((A\oplus B,\dashv_d,\vd_d,\omega ),(A,\dva,\vda), (B, \dashv_{B}, \vdb))$ of diassociative algebras induces a Manin triple  $((A\oplus B,\circ_d,\omega ),(A, \circ_A), (B, \circ_{B}))$ of Leibniz algebras, where the operations $\circ_d, \circ_A$ and $\circ_{B}$ are given through ~\eqref{eq:dil}.
In particular, a standard Manin triple of diassociative algebras induces a standard Manin triple of Leibniz algebras in the same way.
\end{prop}
\begin{proof}
This follows from Example~\mref{exam:dia-Lei} and Proposition~\mref{prop:dia-ind-Lei}.
\end{proof}
\begin{defn}\cite{BLST25,TS22}
Let $(A,\circ_{A})$ and $(A^{*},\circ_{A^{*}})$ be Leibniz algebras. Let $\Delta_{A} :A \to A\ot A$ be the linear dual of $\circ_{A^{*}}$. Then $(A,\circ_{A},\Delta_{A})$ is called
 a {\bf Leibniz bialgebra} if the following conditions hold: for all $x,y\in A$,
\begin{equation*}
\sigma((R_{\circ_{A}}(y)\ot \id)\Delta_{A} (x))=(R_{\circ_{A}}(x)\ot \id)\Delta_{A} (y),
\end{equation*}
\begin{equation*}
\Delta_{A} (x\circ_{A}y)= (\id\ot R_{\circ_{A}}(y)-L_{\circ_{A}}(y)\ot \id-R_{\circ_{A}}(y)\ot\id) (\id+\sigma) \Delta_{A} (x)+(\id\ot L_{\circ_{A}}(x)+L_{\circ_{A}}(x)\ot \id)\Delta_{A} (y).
\end{equation*}
\end{defn}
Then by \cite[Theorem 2.14]{TS22}, we obtain
\begin{theorem} \mlabel{thm:Lei-equiv}
Let $(A,\circ_{A})$ and $(A^*, \circ_{A^*})$ be Leibniz algebras. Let $\Delta_A $ be the linear dual of $\circ_{A^*}$. Then the following
 conditions are equivalent:
 \begin{enumerate}
\item
$(A,\circ_{A},\Delta_A)$ is a Leibniz bialgebra.
\item There is a standard Manin triple of Leibniz algebras
$((A\oplus A^{*},\circ_{d},\omega_{d}), (A,\circ_{A}),(A^{*},\circ_{A^{*}}))$, where the invariant skew
symmetric bilinear form $\omega_d$  on $\frakg\oplus \frakg^*$ is given by ~\eqref{eq:omd}.
\end{enumerate}
 \end{theorem}

\begin{theorem}
\mlabel{thm:dabi-leibi}
	Let $(A,\dva,\vda)$ and $(A^*,\dashv_{A^{*}},\vds)$ be \das. Let $\Delta_{\dva}$ and $\Delta_{\vda}$ be the linear duals of $\dashv_{A^{*}}$ and  $\vds$, respectively. If  the $5$-tuple $(A,\dva,\vda,\Delta_{\dva},\Delta_{\vda})$  is a \dabi, then it induces a Leibniz bialgebra $(A, \circ_{A}, \Delta)$,  where $\circ_{A}$ and $\Delta$  are given by ~\eqref{eq:dil} and ~\meqref{eq:dia-coproduct}, respectively.
\end{theorem}
\begin{proof} This follows from Corollary~\mref{coro:equiv}, Proposition~\mref{prop:dia-manin-lei} and Theorem~\mref{thm:Lei-equiv}.
\end{proof}

\subsection{Construction of Lie bialgebras via diassociative bialgebras and quadratic dendriform algebras}

\begin{defn}
A {\bf dendriform algebra} is a triple $(A, \prec_{A}, \succ_{A})$ consisting of a vector space $A$ and two bilinear operations $\prec_{A}, \succ_{A}: A \otimes A \to A$ such that for all $x, y, z \in A$:
\begin{eqnarray}\mlabel{eq:dendr}
&&(x \prec_{A} y) \prec_{A} z = x \prec_{A} (y \ast_{A} z), \quad\\
&&(x \succ_{A} y) \prec_{A} z = x \succ_{A} (y \prec_{A} z), \quad\\
&&(x \ast_{A} y) \succ_{A} z = x \succ_{A} (y \succ_{A} z),
\end{eqnarray}
where 
\begin{equation}\mlabel{eq:ast}
x \ast_{A} y := x \prec_{A} y + x \succ_{A} y.
\end{equation}
\end{defn}
For any dendriform algebra $(A,\prec_{A},\succ_{A})$,  it is well known \cite{Lo1} that the product $\ast_{A}=\prec_{A}+\succ_{A}$ defines an associative algebra $(A, \ast_{A})$, which is called the {\bf sub-adjacent associative algebra} of $(A,\prec_{A},\succ_{A})$. Conversely,  $(A,\prec_{A},\succ_{A})$ is called the {\bf compatible dendriform algebra} of $(A,\ast_{A})$.

\begin{defn}
\cite{Bai,Cha05}
 A {\bf quadratic dendriform algebra} is a dendriform algebra $(A,\prec_{A}, \succ_{A})$ equipped with a nondegenerate antisymmetric bilinear form $\omega$  that is {\bf invariant} in the following sense: for all $ x,y,z\in A$,
\begin{equation}\mlabel{eq:qda}
\omega(x\prec_{A} y,z)=\omega(x,y\ast_{A} z),\qquad\omega(x\succ_{A} y,z)=\omega(y,z\ast_{A} x).
\end{equation}
\end{defn}
\begin{prop}\mlabel{prop:dend-ass}\cite{Bai}
	Let $(A, \prec_{A}, \succ_{A},\omega)$ be a quadratic dendriform algebra. Then the bilinear form $\omega$ induces a Connes cocycle on the sub-adjacent associative algebra $(A, \ast_{A})$, where $\ast_{A}$ is given by ~\meqref{eq:ast}.
	Conversely, let $(A, \ast_{A})$ be an associative algebra equipped with a nondegenerate Connes cocycle $\omega$. Then there exists a compatible quadratic dendriform algebra structure $(A, \prec_{A}, \succ_{A})$ on $A$, where the operations $\prec_{A}, \succ_{A}$ are uniquely determined by  \meqref{eq:qda}.
	Hence there is a one-to-one correspondence between quadratic dendriform algebras and associative algebras with nondegenerate Connes cocycles.
\end{prop}

Recall that a {\bf quadratic Lie algebra} is a Lie algebra $(\frak g,[-,-]_{\frak g})$ together with a nondegenerate symmetric bilinear form $\mathcal{B}$ which is invariant in the sense that
\begin{equation*}
	\mathcal{B}([x,y]_{\frak g},z)=\mathcal{B}(x,[y,z]_{\frak g}),\; x,y,z\in \frak g.
\end{equation*}

\begin{prop}\mlabel{prop:dend-dia}
Let $(A,\prec_{A},\succ_{A})$ be a dendriform algebra and $(B,\dashv_{B},\vdb)$  a diassociative algebra. Then the tensor product $\frak g=A\ot B$ of $A$ and $B$ admits a Lie algebra structure with the bracket defined by
	\begin{align}\mlabel{eq:xyab}
	[x\otimes a,y\otimes b]_{\frak g}:=x\succ_{A}y\otimes a\vdb b+x\prec_{A}y\otimes a\dashv_{B}b-y\succ_{A}x\otimes b\vdb a-y\prec_{A}x\otimes b\dashv_{B}a.
	\end{align}
In particular, if there exists bilinear forms $\omega_{1},\omega_{2}$ such that  $(A,\succ_{A},\prec_{A},\omega_{1})$ is a quadratic dendriform algebra  and   
$(B,\dashv_{B},\vdb,\omega_{2})$ is a  quadratic  diassociative algebra, then there is a 
 a quadratic  Lie algebra $(\frak g,[-,-]_{\frak g},\mathcal{B} )$, where
\begin{eqnarray}\mlabel{eq:oxy}
	\mathcal{B}(x\otimes a,y\otimes b)=\omega_{1}(x,y)\omega_{2}(a,b),\quad x,y\in A, a,b\in B.
\end{eqnarray}
\end{prop}

\begin{proof}
Similar to the proof of ~\cite[Proposition~5.3, P37]{Lo1}, $(\frak g,[-,-]_{\frak g})$ is a Lie algebra.
It suffices to verify that the Lie bracket given by ~\meqref{eq:xyab} and $\mathcal{B}$ defined by ~\meqref{eq:oxy} satisfy
\begin{equation}\mlabel{eq:invar}
\mathcal{B}([x\ot a,y\ot b]_{\frak g},z\ot c)=\mathcal{B}(x\ot a,[y\ot b,c\ot z]_{\frak g}),\quad x,y,z\in A, a,b,c\in B.
\end{equation}
By ~\meqref{eq:omega} and \meqref{eq:qda}, we have
\begin{align*}
&\mathcal{B}([x\ot a,y\ot b]_{\frak g},z\ot c)\\
&=\mathcal{B}(x\succ_{A}y\otimes a\vdb b+x\prec_{A}y\otimes a\dashv_{B}b-y\succ_{A}x\otimes b\vdb a-y\prec_{A}x\otimes b\dashv_{B}a, z\ot c)\\
&\stackrel{\meqref{eq:oxy}}{=}\omega_1(x\succ_A y,z)\omega_2(a\vdb b,c)+\omega_1(x\prec_A y,z)\omega_2(a\dashv_B b,c)\\
&\quad-\omega_1(y\succ_A x, z)\omega_2( b\vdb a, c)-\omega_1(y\prec_A x,z)\omega_2(b\dashv_B a, c)\\
&\stackrel{\meqref{eq:omega}\meqref{eq:qda}}{=}-\omega_1(x,y\prec_A z)\omega_1(a\vdb b,c)+(\omega_1(x,y\prec_A z)+\omega_1(x,y\succ_A z))\omega_2(a\dashv_B b,c)\\
&+\omega_1(z,y\succ_A x)\omega_2(a, c\dashv_B b)+\omega_1(z,y\prec_A x)(\omega_2(c\vdb b,a)-\omega_2(c\dashv_B b,a)).
\end{align*}
On the other hand, we get
\begin{align*}
&\mathcal{B}(x\ot a,[y\ot b,z\ot c]_{\frak g})\\
&=\mathcal{B}(x\ot a, y\succ_{A}z\otimes b\vdb c+y\prec_{A}z\otimes b\dashv_{B}c
-z\succ_{A}y\otimes c\vdb b-z\prec_{A}y\otimes c\dashv_{B}b)\\
&\stackrel{\meqref{eq:oxy}}{=}\omega_1(x,y\succ_A z)\omega_2(a, b\vdb c)+\omega_1(x,y\prec_A z)\omega_2(a, b\dashv_B c)\\
&\quad-\omega_1(x,z\succ_A y)\omega_2(a, c\vdb b)-\omega_1(x,z\prec_A y)\omega_2(a, c\dashv_B b)\\
&\stackrel{\meqref{eq:omega}\meqref{eq:qda}}{=}\omega_1(x,y\succ_A z)\omega_2(a\dashv_B b, c)+\omega_1(x,y\prec_A z)(-\omega_2(a\vdb b,c)+\omega_2(a\dashv_B b,c))\\
&+\omega_1(z, y\prec_A x)\omega_2( c\vdb b,a)+(\omega_1(z,y\prec_A x)+\omega_1(z,y\succ_A x))\omega_2(a, c\dashv_B b).
\end{align*}
By comparing the two equations above, we obtain ~\meqref{eq:invar}.
\end{proof}

\begin{defn}\mlabel{defn:triple-Lie}\mcite{CP94}
	A {\bf Manin triple of Lie algebras} is a collection $((\frak g\oplus\frak h,[-,-]_{d},\mathcal{B}),(\frak g$,
	$
	[-,-]_{\frak g}), 
(\frak h,[-,-]_{\frak h}))$ such that $(\frak g\oplus\frak h,[-,-]_{d},\mathcal{B})$  is a quadratic Lie algebra, $(\frak g,[-,-]_{\frak g}),(\frak h,[-,-]_{\frak h})$ are Lie subalgebras of $(\frak g\oplus\frak h,[-,-]_{d})$ and $\frak g,\frak h$ are isotropic with respect to $\mathcal{B}$.
		A Manin triple of Lie algebras is called {\bf standard} if $\frak h=\frak g^{*}$ and $\mathcal{B}=\mathcal{B}_{s}$ is given by
		\begin{equation}\mlabel{2250}
			\mathcal{B}_{s}(x+a^{*},y+b^{*})=\langle x,b^{*}\rangle+\langle a^{*},y\rangle,\; x,y\in\frak g, a^{*},b^{*}\in \frak g^{*}.
		\end{equation}
\end{defn}

	Two Manin triples of Lie algebras $((\frak g\oplus\frak h,[-,-]_{d},\mathcal{B}),(\frak g,[-,-]_{\frak g}),(\frak h,[-,-]_{\frak h}))$ and $((\frak g'\oplus\frak h'$,
	$[-,-]'_{d},\mathcal{B}'), 
(\frak g',[-,-]_{\frak g'}),(\frak h',[-,-]_{\frak h'}))$ are called {\bf isomorphic} if there exists an isomorphism $f:\frak g\oplus\frak h\rightarrow\frak g'\oplus\frak h'$ of Lie algebras such that
\begin{equation*}
	f(\frak g)\subset \frak g',\;
	f(\frak h)\subset \frak h',\;
	\mathcal{B}'(f(x+a),f(y+b))=\mathcal{B}(x+a,y+b),\; x,y\in\frak g, a,b\in\frak h.
\end{equation*} 
As proved in \mcite{CP94}, every Manin triple of Lie algebras is isomorphic to a standard one.
\begin{defn}\mcite{CP94} A {\bf Lie bialgebra} is a triple $(\mathfrak{g},[-,-]_{\frak g},\delta)$ such that $(\mathfrak{g},[-,-]_{\frak g})$ is a Lie algebra,
	$(\mathfrak{g},\delta)$ is a Lie coalgebra, and
	$\delta$ is a 1-cocycle of $\mathfrak{g}$ with values in $\mathfrak{g}\otimes \mathfrak{g}$, that is,
	\begin{equation*}
		\delta([x,y]_{\frak g})=(\mathrm{ad}_{\frak g}(x)\otimes \mathrm{id}+\mathrm{id}\otimes\mathrm{ad}_{\frak g}(x))\delta(y)-(\mathrm{ad}_{\frak g}(y)\otimes \mathrm{id}+\mathrm{id}\otimes\mathrm{ad}_{\frak g}(y))\delta(x),\quad x,y\in \mathfrak{g}.
	\end{equation*}
\end{defn}

\begin{theorem}\mcite{CP94}\mlabel{thm:Liebia}
Let $(\frakg,[-,-]_\frakg)$ and $(\frakg^*, [-,-]_{\frakg^*})$ be Lie algebras. Let $\delta$ be the linear dual of $[-,-]_{\frakg^*}$.
Then the following conditions are equivalent:
\begin{enumerate}
\item There is a standard Manin triple of Lie algebras 
$((\mathfrak{g}\oplus \mathfrak{g}^{*},[-,-]_{d}, \mathcal{B}_{s}),(\mathfrak{g},[-,-]_{\frak g}),(\mathfrak{g}^{*}$,
$[-,-]_{\frak g^{*}}))$.
\item
$(\frakg, [,]_\frakg,\delta)$ is a Lie bialgebra.
\end{enumerate}
\end{theorem}

\begin{prop}\mlabel{2283}
	Let $(A,\prec_{A},\succ_{A},\omega)$ be a quadratic dendriform algebra. Let $(B,\dashv_{B},\vdb)$ and $(B^{*},\dashv_{B^{*}},\vdbs)$ be diassociative algebras, and $(A\otimes B,[-,-]_{A\otimes B})$ and $(A\otimes B^{*},[-,-]_{A\otimes B^{*}})$ be the   Lie algebras induced from  $(A,\prec_{A},\succ_{A} )$ and  $(B,\dashv_{B},\vdb)$
	 as well as $(A,\prec_{A},\succ_{A} )$ and  $(B^{*},\dashv_{B^{*}},\vdbs)$ respectively.
	 If there is a standard Manin triple of diassociative algebras $( (B\oplus B^{*},\dashv_{d},\vd_{d},\omega_{d}),  (B,\dashv_{B},\vdb),(B^{*},\dashv_{B^{*}},\vdbs) )$, then there is a Manin triple of Lie algebras
	 $((A\otimes (B\oplus B^{*}),[-,-]_{d},\mathcal{B}), (A\otimes B,[-,-]_{A\otimes B}), (A\otimes B^{*},[-,-]_{A\otimes B^{*}}))$, where 
	 \begin{eqnarray*}
&&[x\otimes a+y\otimes b^{*}, z\otimes c+w\otimes d^{*}]_{d}\\
&=&x\succ_{A} z\otimes a\vdb c+x\prec_{A}z\otimes a\dashv_{B}c
-z\succ_{A} x\otimes c\vdb a-z\prec_{A}x\otimes c\dashv_{B}a\\
&&+x\succ_{A} w\otimes a\vd_{d}d^{*}+x\prec_{A}w\otimes a\dashv_{d}d^{*}
-w\succ_{A} x\otimes d^{*}\vd_{d}a-w\prec_{A}x\otimes d^{*}\dashv_{d}a\\
&&+y\succ_{A}z\otimes b^{*}\vd_{d}c+y\prec_{A}z\otimes b^{*}\dashv_{d}c-z\succ_{A}y\otimes c\vd_{d}b^{*}-z\prec_{A}y\otimes c\dashv_{d}b^{*}\\
&&+y\succ_{A}w\otimes b^{*}\vdbs d^{*}+y\prec_{A}w\otimes b^{*}\dashv_{B^{*}}d^{*}
-w\succ_{A}y\otimes d^{*}\vdbs b^{*}-w\prec_{A}y\otimes d^{*}\dashv_{B^{*}}b^{*}
	 \end{eqnarray*}
 and 
 \begin{equation*}
 	\mathcal{B}(x\otimes a+y\otimes b^{*},z\otimes c+w\otimes d^{*})=\omega(y,z)\langle b^{*},c\rangle-\omega(x,w)\langle a,d^{*}\rangle,
 \end{equation*}
for all $x,y,z,w\in A, a,c\in B, b^{*},d^{*}\in B^{*}$.
	\end{prop}
\begin{proof}
	It follows directly from Proposition \mref{prop:dend-dia}.
\end{proof}

Keep the same  assumptions in Proposition \mref{2283}.
By the nondegeneracy of $\omega$ in the quadratic dendriform algebra $(A,\prec_{A},\succ_{A},\omega)$, there is a linear isomorphism $\varphi:A\rightarrow A^{*}$ defined by
\begin{equation}\mlabel{2322}
	\langle \varphi(x),y\rangle=\omega(x,y),\; x,y\in A.
\end{equation}
Thus there is a dendriform algebra $(A^{*},\prec_{A^{*}},\succ_{A^{*}})$ defined by 
\begin{equation*}
	a^{*}\prec_{A^{*}}b^{*}=\varphi( \varphi^{-1}(a^{*})\prec_{A} \varphi^{-1}(b^{*})  ),\;
		a^{*}\succ_{A^{*}}b^{*}=\varphi( \varphi^{-1}(a^{*})\succ_{A} \varphi^{-1}(b^{*})  ),\; a^{*},b^{*}\in A^{*}.
\end{equation*}
The two dendriform algebras $(A,\prec_{A},\succ_{A})$ and $(A^{*},\prec_{A^{*}},\succ_{A^{*}})$ are isomorphic through $\varphi$.
Moreover, we obtain the induced Lie algebra $(\frak g^{*}=A^{*}\otimes B^{*},[-,-]_{\frak g^{*}})$ from $(A^{*},\prec_{A^{*}},\succ_{A^{*}})$ and $(B^{*},\dashv_{B^{*}},\vdbs)$.
Further, we can define a linear map $f: A\otimes (B\oplus B^{*})\rightarrow (A\otimes B)\oplus (A^{*}\otimes B^{*})$ by
\begin{equation*}
	f(x\otimes a+y\otimes b^{*})=x\otimes a+\varphi(y)\otimes b^{*},\; x,y\in A, a\in B, b^{*}\in B^{*}.
\end{equation*}
Evidently, $f$ is a linear isomorphism. Thus we can transport the Lie algebra structure on $A\otimes (B\oplus B^{*})$ to $(A\otimes B)\oplus (A^{*}\otimes B^{*})$ by 
\begin{equation}\mlabel{2331}
	[x\otimes a+y^{*}\otimes b^{*},z\otimes c+w^{*}\otimes d^{*}]'_{d}=f[x\otimes a+\varphi^{-1}(y^{*})\otimes b^{*}, z\otimes c+\varphi^{-1}(w^{*})\otimes d^{*}]_{d},
\end{equation}
for all $x,z\in A, y^{*},w^{*}\in A^{*}, a,c\in B, b^{*},d^{*}\in B^{*}$.
Moreover, the restrictions of the Lie bracket $[-,-]'_{d}$ on $\frak g=A\otimes B$ and $\frak g^{*}=A^{*}\otimes B^{*}$ coincide with $[-,-]_{\frak g}$ and $[-,-]_{\frak g^{*}}$ respectively.
It is also straightforward to check that the natural nondegenerate symmetric bilinear form $\mathcal{B}_{s}$ given by  \meqref{2250} is invariant on $(\frak g\oplus\frak g^{*},[-,-]'_{d})$.
Hence we have the following result.

\begin{lemma}\mlabel{2337}
With the notations above, $((\frak g\oplus \frak g^{*},[-,-]'_{d},\mathcal{B}_{s}),(\frak g,[-,-]_{\frak g}),(\frak g^{*},[-,-]_{\frak g^{*}}))$ is a standard Manin triple of Lie algebras. Moreover, the two Manin triples of Lie algebras  $((A\otimes (B\oplus B^{*}),[-,-]_{d},\mathcal{B}), (A\otimes B,[-,-]_{A\otimes B}), (A\otimes B^{*},[-,-]_{A\otimes B^{*}}))$ and $((\frak g\oplus \frak g^{*},[-,-]'_{d},\mathcal{B}_{s}),(\frak g,[-,-]_{\frak g}),(\frak g^{*}$,
$[-,-]_{\frak g^{*}}))$ are isomorphic through $f$.
\end{lemma}

Combining Corollary \mref{coro:equiv}, Proposition \mref{2283}, Lemma \mref{2337} and Theorem \mref{thm:Liebia} together, we have the following result, which derives Lie bialgebras from quadratic dendriform algebras and  diassociative bialgebras.

\begin{prop}\mlabel{prop:dend-dia-Lie}
	Let $(A,\prec_{A},\succ_{A},\omega)$ be a quadratic dendriform algebra. Let $(B, \dashv_{B},\vdb,\Delta_{\dvb},\Delta_{\vdb})$ be a diassociative bialgebra and $((B\oplus B^{*},\dashv_{d},\vd_{d},\omega_{d}), (B,\dashv_{B},\vdb),(B^{*},\dashv_{B^{*}},\vdbs))$ be the corresponding standard Manin triple of diassociative algebras. Then there is a standard Manin triple of Lie algebras $((\frak g\oplus \frak g^{*},[-,-]'_{d},\mathcal{B}_{s}),(\frak g,[-,-]_{\frak g}),(\frak g^{*},[-,-]_{\frak g^{*}}))$, where $\frak g=A\otimes B$, $[-,-]'_{d}$ is given by \meqref{2331} and $\mathcal{B}_{s}$ is given by \meqref{2250}. Consequently, there is a Lie bialgebra $(\frak g,[-,-]_{\frak g},\delta)$, where $\delta$ is the linear dual of $[-,-]_{\frak g^{*}}$.
\end{prop}

\begin{remark}
	Similar to the above discussions, the tensor product of a quadratic diassociative algebra and a dendriform D-bialgebra \mcite{Bai} also yields a Manin triple of Lie algebras. Moreover, we can also derive Lie bialgebras from quadratic diassociative algebras and  dendriform D-bialgebras.
\end{remark}

\subsection{Construction of diassociative bialgebras via  ASI bialgebras and quadratic perm algebras}
\begin{defn}  A {\bf perm algebra} \cite{Cha01} is a pair $(A,\circ_{A})$ consisting of a vector space $A$ and a binary operation $\circ_{A}: A\ot A\to A$ such that
\begin{equation*}
(x\circ_{A} y)\circ_{A} z=x\circ_{A} (y\circ_{A} z)=x\circ_{A} (z\circ_{A} y),\quad x,y,z\in A.
\end{equation*}
A {\bf quadratic perm algebra} \cite{Cha05} is a triple $(A,\circ_{A},\omega)$ such that $(A,\circ_{A})$ is a perm algebra and $\omega$ is a nondegenerate antisymmetric bilinear form on $(A,\circ_{A})$ which is invariant in the sense that 
\begin{equation}\mlabel{eq:perm-inv}
	\omega(x\circ_{A} y,z)=\omega(y,z\circ_{A}  x-x\circ_{A}  z) ,\quad x,y,z\in A.
\end{equation}
\end{defn}
 
Recall that a (symmetric) Frobenius  algebra is an associative algebra $(A,\cdot_{A})$ together with a nondegenerate (symmetric) bilinear form $\mathcal{B}$ which is invariant in the sense that
\begin{equation}\label{2426}
\mathcal{B}(x\cdot_{A}y,z)=\mathcal{B}(x,y\cdot_{A}z),\; x,y,z\in A.
\end{equation}
 
\begin{prop} \mlabel{prop:ass-perm} Let $(A,\cdot_{A})$ be an associative algebra  and let $(B,\circ_{B})$ be a perm algebra. 
	Then there is a diassociative algebra structure on $C=A\otimes B$ given by
\begin{equation}\mlabel{eq:ap}
(x\ot a)\dashv_{C} (y\ot b):=(x\cdot_{A} y)\ot (a\circ_{B} b),\, (x\ot a)\vd_{C} (y\ot b):=(x\cdot_{A} y)\ot (b\circ_{B} a),\, x,y\in A, a,b\in B.
\end{equation}
Furthermore, if there are bilinear forms $\mathcal{B}$ and $\omega$ on $A$ and $B$ such that $(A,\cdot_{A},\mathcal{B})$ is a symmetric Frobenius algebra and $(B,\circ_{B},\omega )$ is a quadratic perm algebra, then there is a  quadratic diassociative algebra $(C=A\ot B,\dashv_{C},\vd_{C},\varpi)$, where $\dashv_{C}$ and $\vd_{C}$ are given by ~\meqref{eq:ap}, and $\varpi$ is defined by
\begin{equation*}
\varpi(x\ot a, y\ot b):=\mathcal{B}(x,y)\omega(a,b),\quad x,y\in A, a,b\in B.
\end{equation*}
\end{prop}
\begin{proof} 
The first half part follows from a straightforward computation. To prove the second half part, 
it suffices to verify that $\varpi$ satisfies ~\meqref{eq:omega}, that is, for all $x,y,z\in A, a,b,c\in B$,
\begin{eqnarray*}
	\varpi(x\ot a,(y\ot b)\dashv_{C} (z\ot c))&=&\varpi(z\ot c, (x\ot a)\ob_{C} (y\ot b)),
		\\
		\varpi((x\ot a)\vd_{C} (y\ot b),z\ot c)&=&\varpi(x\ot a, (y\ot b)\ob_{C} (z\ot c)).
	\end{eqnarray*}
By \eqref{2426} and \meqref{eq:perm-inv}, we have
\begin{eqnarray*}
\varpi(z\ot c, (x\ot a)\ob_{C} (y\ot b))&=&\varpi(z\ot c, x\cdot_{A} y\ot (b\circ_{B} a- a\circ_{B} b))\\
&=&\mathcal{B}(z,x\cdot_{A}y)\omega(c,b\circ_{B} a-a\circ_{B} b)\\
&=&\mathcal{B}(x,y\cdot_{A}z)\omega(a, b\circ_{B} c)\\
&=&\varpi(x\ot a,(y\ot b)\dashv_{C} (z\ot c))
\end{eqnarray*}
A similar proof applies to the other equation.
\end{proof}

\begin{defn}\mcite{Bai}
A {\bf double construction of Frobenius algebras} is a collection $((A\oplus A^{*},\cdot_{d}$,
$\mathcal{B}_{s}),
(A,\cdot_{A}),(A^{*},\cdot_{A^{*}}))$ such that $(A\oplus A^{*},\cdot_{d},\mathcal{B}_{s})$ is a symmetric Frobenius algebra and $(A,\cdot_{A}),(A^{*},\cdot_{A^{*}})$ are associative subalgebras of $(A\oplus A^{*},\cdot_{d})$.
\end{defn}

\begin{defn}\cite{Bai}
\mlabel{de:bial}
An  {antisymmetric infinitesimal bialgebra} (or an {\bf ASI bialgebra} in short) is a triple $(A,\cdot_{A},\Delta)$ such that $(A,\cdot_{A})$ is an associative algebra, $(A,\Delta)$ is a  coassociative coalgebra and the following equations hold:
 \begin{equation}
 \Delta(x\cdot_{A} y)=(R_{\cdot_{A}}(y)\otimes \id)\Delta(x)+(\id\otimes L_{\cdot_{A}}(x))\Delta(y),
 \mlabel{eq:3.14}
 \end{equation}
 \begin{equation}
 (L_{\cdot_{A}}(x)\otimes \id-\id\otimes R_{\cdot_{A}}(x))\Delta(y)=\sigma((\id\otimes R_{\cdot_{A}}(y)-L_{\cdot_{A}}(y)\otimes \id)\Delta(x)), \quad  x, y\in A.
 \mlabel{eq:3.15}
 \end{equation}
\end{defn}

\begin{prop}\cite{Bai}\mlabel{2486}
Let $(A,\cdot_{A})$ and $(A^{*},\cdot_{A^{*}})$ be associative algebras and $\Delta:A\rightarrow A\otimes A$ be the linear dual of $\cdot_{A^{*}}$. Then $(A,\cdot_{A},\Delta)$ is an ASI bialgebra if and only if there is a double construction of Frobenius algebras $((A\oplus A^{*},\cdot_{d},\mathcal{B}_{s}),(A,\cdot_{A}),(A^{*},\cdot_{A^{*}}))$.
\end{prop}

\begin{prop}\mlabel{2490}
	Let $(A,\cdot_{A})$ and $(A^{*},\cdot_{A^{*}})$ be associative algebras and $(B,\circ_{B},\omega)$ be a quadratic perm algebra.
	Let $(C=A\otimes B,\dashv_{C},\vd_{C})$ and $(D=A^{*}\otimes B,\dashv_{D},\vd_{D})$ be the diassociative algebras induced from $(A,\cdot_{A})$ and $(B,\circ_{B})$ as well as $(A^{*},\cdot_{A^{*}})$ and $(B,\circ_{B})$ respectively.
	If there is a double construction of Frobenius algebras $((A\oplus A^{*},\cdot_{d},\mathcal{B}_{s}),(A,\cdot_{A}),(A^{*},\cdot_{A^{*}}))$, then there is a Manin triple of diassociative algebras $((C\oplus D,\dashv_{d},\vd_{d},\varpi),(C,\dashv_{C},\vd_{C}), (D,\dashv_{D},\vd_{D}))$, where
	\begin{eqnarray*}
&&(x\otimes a+y^{*}\otimes b)\dashv_{d}(z\otimes c+w^{*}\otimes d)\\
&&=x\cdot_{A}z\otimes a\circ_{B}c+x\cdot_{d}w^{*}\otimes a\circ_{B}d+y^{*}\cdot_{d}z\otimes b\circ_{B}c+y^{*}\cdot_{A^{*}}w^{*}\otimes b\circ_{B}d,\\
&&(x\otimes a+y^{*}\otimes b)\vd_{d}(z\otimes c+w^{*}\otimes d)\\
&&=x\cdot_{A}z\otimes c\circ_{B}a+x\cdot_{d}w^{*}\otimes d\circ_{B}a+y^{*}\cdot_{d}z\otimes c\circ_{B}b+y^{*}\cdot_{A^{*}}w^{*}\otimes d\circ_{B}b
	\end{eqnarray*}
and 
\begin{equation*}
	\varpi(x\otimes a+y^{*}\otimes b,z\otimes c+w^{*}\otimes d)=\langle x,w^{*}\rangle\omega(a,d)+\langle y^{*},z\rangle\omega(b,c)
\end{equation*}
for all $x,z\in A, y^{*},w^{*}\in A^{*}, a,b,c,d\in B$.
\end{prop}
\begin{proof}
It follows directly from Proposition \mref{prop:ass-perm}.	
\end{proof}

Keep the same assumptions in Proposition \mref{2490}.
By the nondegeneracy of $\omega$ in the quadratic perm algebra $(B,\circ_{B},\omega)$, there is a linear isomorphism $\varphi: B\rightarrow B^{*}$ given by
\begin{equation*}
	\langle \varphi(a),b\rangle=\omega(a,b),\; a, b\in B.
\end{equation*}
Thus there is a perm algebra $(B^{*},\circ_{B^{*}})$ defined by
\begin{equation*}
	a^{*}\circ_{B^{*}}b^{*}=\varphi( \varphi^{-1}(a^{*})\circ_{B} \varphi^{-1}(b^{*}) ),\; a^{*},b^{*}\in B^{*}.
\end{equation*}
The two perm algebras $(B,\circ_{B} )$ and $(B^{*},\circ_{B^{*}})$ are isomorphic through $\varphi$.
Moreover, we obtain the induced diassociative algebra $(C^{*}=A^{*}\otimes B^{*},\dashv_{C^{*}},\vd_{C^{*}})$ from $(A^{*},\cdot_{A^{*}})$ and $(B^{*},\circ_{B^{*}})$.
Further, we can define a linear map $f:(A\oplus A^{*})\otimes B\rightarrow (A\otimes B)\oplus (A^{*}\otimes B^{*})$ by
\begin{equation*}
	f(x\otimes a+y^{*}\otimes b)=x\otimes a+y^{*}\otimes \varphi(b),\; x\in A, y^{*}\in A^{*}, a,b\in B.
\end{equation*}
Evidently, $f$ is a linear isomorphism.
Thus we can transport the diassociative algebra structure on $(A\oplus A^{*})\otimes B$ to $ (A\otimes B)\oplus (A^{*}\otimes B^{*}) $ by
\begin{eqnarray}
 ( x\otimes a+y^{*}\otimes b^{*}  )\dashv'_{d}(z\otimes c+w^{*}\otimes d^{*})=f(  ( x\otimes a+y^{*}\otimes \varphi^{-1}(b^{*})  )\dashv_{d}(z\otimes c+w^{*}\otimes \varphi^{-1}(d^{*}) )  ),\mlabel{2528}&&\\
( x\otimes a+y^{*}\otimes b^{*}  )\vd'_{d}(z\otimes c+w^{*}\otimes d^{*})=f(  ( x\otimes a+y^{*}\otimes \varphi^{-1}(b^{*})  )\vd_{d}(z\otimes c+w^{*}\otimes \varphi^{-1}(d^{*}) )  ),\mlabel{2529}&&
\end{eqnarray}
for all $x,z\in A, y^{*},w^{*}\in A^{*}, a,c\in B, b^{*}\in B^{*}$.
Moreover, the restrictions of $ \dashv'_{d},\vd'_{d} $ on $C=A\otimes B$ and $C^{*}=A^{*}\otimes B^{*}$ coincide with $  \dashv_{C}, \vd_{C}$ and $\dashv_{C^{*}}, \vd_{C^{*}}$ respectively.
It is also straightforward to check that the natural nondegenerate antisymmetric bilinear form $\omega_{d}$ given by \meqref{eq:omd} is invariant on $(C\oplus C^{*},\dashv'_{d},\vd'_{d})$.
Note that the definition of the isomorphism of Manin triples of diassociative algebras is given similarly to that of Manin triple of  Lie algebras (and every Manin triple of diassociative algebras is isomorphic to a standard one).
Hence we have the following result.

\begin{lemma}\mlabel{2537}
	With the notations above, $( (C\oplus C^{*},\dashv'_{d},\vd'_{d},\omega_{d}),(C,\dashv_{C},\vd_{C}),(C^{*},\dashv_{C^{*}},\vd_{C^{*}}) )$ is a standard Manin triple of diassociative algebras. Moreover, the two Manin triples of diassociative  algebras $((C\oplus D,\dashv_{d},\vd_{d},\varpi),(C,\dashv_{C},\vd_{C}), (D,\dashv_{D},\vd_{D}))$ and $( (C\oplus C^{*},\dashv'_{d},\vd'_{d},\omega_{d}),(C,\dashv_{C},\vd_{C}),(C^{*},\dashv_{C^{*}},\vd_{C^{*}}) )$  are isomorphic through $f$.
\end{lemma}

Combining Corollary \mref{coro:equiv}, Proposition \mref{2486},
Proposition \mref{2490} and 
 Lemma \mref{2537} together, we have the following result, which derives diassociative bialgebras from quadratic perm algebras and  ASI bialgebras.

\begin{prop}\mlabel{prop:ass-perm-dia}
	Let $(A,\cdot_{A},\Delta)$ be an ASI bialgebra and $((A\oplus A^{*},\cdot_{d},\mathcal{B}_{s}),(A,\cdot_{A}), (A^{*},\cdot_{A^{*}}))$ be the corresponding double construction of Frobenius algebras.
	Let $(B,\circ_{B},\omega)$ be a quadratic perm algebra.
	 Then there is a standard Manin triple of diassociative algebras $( (C\oplus C^{*},\dashv'_{d},\vd'_{d},\omega_{d}),(C,\dashv_{C},\vd_{C}),(C^{*},\dashv_{C^{*}},\vd_{C^{*}}) )$, where $C=A\otimes B$, $\dashv'_{d},\vd'_{d}$ are given by \meqref{2528}-\meqref{2529}  and $\omega_{d}$ is given by \meqref{eq:omd}. Consequently, there is a diassociative bialgebra $(C,\dashv_{C},\vd_{C},\Delta_{\dv_{C}},\Delta_{\vd_{C}})$, where $\Delta_{\dv_{C}},\Delta_{\vd_{C}}$ are the linear duals of $\dashv_{C^{*}},\vd_{C^{*}}$ respectively.
\end{prop}

\begin{remark}
	Similar to the above discussions, the tensor product of a symmetric Frobenius algebra and a perm bialgebra \mcite{Hou24, Lin} also yields a Manin triple of diassociative algebras. Moreover, we can also derive diassociative bialgebras from symmetric Frobenius algebras and  perm bialgebras.
\end{remark}

\smallskip

\noindent {\bf Acknowledgements}: This work is supported by NSFC (12461002, 12326324), the China Postdoctoral Science Foundation (2024T005TJ, 
2024M761507), the Postdoctoral Fellowship Program of CPSF (GZC20240755) and Nankai Zhide Foundation.
Shanghua Zheng thanks the Chern Institute of Mathematics at Nankai University for hospitality.

\noindent
{\bf Declaration of interests. } The authors have no conflicts of interest to disclose.

\noindent
{\bf Data availability. } Data sharing is not applicable as no new data were created or analyzed.

\end{document}